\documentclass[11pt]{amsart}

\usepackage{amssymb}
\usepackage{graphicx}
\usepackage{caption} 

\usepackage{amsmath}
\usepackage{amsfonts}
\usepackage{amscd}

\usepackage{amsthm}
\usepackage{marginnote}
\usepackage{enumerate}

\usepackage[small,nohug,heads=vee]{diagrams}
\diagramstyle[labelstyle=\scriptstyle]

\subjclass[2010]{30D05, 37F10}

\newtheorem{theo}{Theorem}[section]

\newtheorem{lemma}{Lemma}[section]
\newtheorem{prop}[theo]{Proposition}

\newtheorem{mainthm}[theo]{Main Theorem}

\theoremstyle{definition}
\newtheorem{defin}[theo]{Definition}

\newtheorem*{xrem}{Remark}
\newtheorem*{xcm}{Claim}

\numberwithin{equation}{section}

\def\cA{\mathcal{A}}
\def\C{\mathbb{C}}
\def\Chat{\hat{\mathbb C}}
\def\D{\mathbb{D}}

\def\S{\mathbb{S}}

\def\cA{\mathcal{A}}
\def\cN{\mathcal{N}}

\def\cF{\mathcal{F}}

\begin{document}
\date{\today}
\title[Postcritically minimal Newton maps]{A characterization of postcritically minimal Newton maps of complex exponential functions}

\author[K. Mamayusupov]{Khudoyor Mamayusupov}
\address{Jacobs University Bremen, Campus Ring 1, 28759, Bremen, Germany}

\address{National Research University
 Higher School of Economics Faculty of Mathematics, Usacheva 6, Moscow, Russia}

\email{k.mamayusupov@jacobs-university.de}

\thanks{Research was partially supported by the ERC advanced grant ``HOLOGRAM" and the Deutsche Forschungsgemeinschaft SCHL 670/4}
\begin{abstract} We obtain a unique, canonical one-to-one correspondence between the space of
	marked postcritically finite Newton maps of polynomials and the space of postcritically minimal Newton maps of entire maps that take the form $p(z) \text{exp}(q(z))$ for $p(z)$, $q(z)$ polynomials and $\text{exp}(z)$, the complex exponential function. This bijection preserves the dynamics and embedding of Julia sets and is induced by a surgery tool developed by Ha\"issinsky.
\end{abstract}
\keywords{Newton's method, basin of attraction, postcritically minimal}
\maketitle
\section{Introduction}
 The Newton map of an entire map $f(z)$ is defined by $N_f(z):=z-f(z)/f'(z)$. The Newton map of $f$ is a rational map only when $f(z)=p(z)\text{exp}(q(z))$, for $p(z)$, $q(z)$ polynomials and $\text{exp}(z)$ ($e^z$ in short) the complex exponential function. Then $N_{pe^q}(z)=z-p(z)/(p'(z)+p(z) q'(z))$. The roots of $p$ are attracting fixed points of $N_{p e^q}$. If $q$ is not constant, a point at $\infty$ is a parabolic fixed point with multiplier $+1$ for $N_{p e^q}$, otherwise $\infty$ is repelling.
 
 \begin{defin}[Postcritically minimal Newton map]\label{Def:PCM} A Newton map $N_{p e^q}$, with polynomials $p$ and $q$, is called postcritically minimal (PCM) if its Fatou set consists of superattracting basins and a parabolic basin of $\infty$, and the following hold.
 \begin{enumerate}[(a)]
 	\item Critical orbits in the Julia set and in superattracting basins are finite;
 	\item Every immediate basin of $\infty$ contains one (possibly with high multiplicity) critical point, and all critical points in a basin of $\infty$ are in \emph{minimal critical orbit relations}: if $c$ is a critical point in a strictly preperiodic component $U$ of the basin of $\infty$, of preperiod $m\ge 1$, then $N^{\circ m}_{pe^q}(c)$ is a critical point in one of the immediate basins of $\infty$.
 \end{enumerate}
  \end{defin}
 
 \begin{mainthm} \label{Thm:Main}
 For every pair of natural numbers $d\ge3$ and $1\le n\le d$, there exists a unique, canonical bijection between the space of Ha\"issinsky equivalent classes of $n$-marked postcritically finite Newton maps of polynomials, of degree $d$, and the space of affine conjugacy classes of degree $d$ postcritically minimal Newton maps of entire maps that take the form $P(z) \text{exp}(Q(z))$ for $P$, $Q$ polynomials with $deg(P)=d-n$ and $\deg(Q)=n$, which preserves dynamics and embedding of Julia sets.
  \end{mainthm}
 
 Marking is defined in Definitions~\ref{Def:MarkedChanneldiagram}~and~\ref{Space:PCF_Marking}, and Ha\"issinsky equivalent classes are defined in Definition~\ref{Def:Surgery_Equivalence}. Recall that a rational map is called postcritically finite (PCF) if its critical orbits are finite: every critical point in the Fatou set eventually terminates at a superattracting periodic point, and every critical point in the Julia set eventually terminates at a repelling periodic point. Moreover, when a superattracting fixed point captures some other critical point, their critical orbit relation is minimal in the sense that the latter lands at the former without wandering within the immediate basin. Otherwise, its orbit is infinite and never lands at the superattracting fixed point since the local dynamics are conjugate to a power map, $z\mapsto z^k$, where $k\ge 2$ is the local degree. 
 
 Relaxing the condition of postcritical finiteness comes at a cost; postcritical minimality is much weaker than postcritical finiteness. Depending on the degree of $q$, there exist $d-1$ distinct `parallel' spaces, of complex dimension $d-2$, of degree $d$ Newton maps of entire maps of the form $P(z)\text{exp}(Q(z))$ for polynomials $P$ and $Q$. In each of these spaces, we distinguish and characterize/classify all postcritically minimal Newton maps. However, we shall not build a parallel theory to the successful theory of classification of postcritically finite Newton maps of polynomials (see \cite{LMS} for the full classification of PCF Newton maps of polynomials). Our goal is to transfer this existing knowledge to our class of rational maps.
 
 The tool used for our characterization is developed by Ha\"issinsky in \cite{Ha} and is called \emph{parabolic surgery}. For the Newton map of a polynomial, this parabolic surgery procedure results in a new rational map, which turns out to be the Newton map of $p(z) \text{exp}(q(z))$, for some polynomials $p(z)$ and $q(z)$  (see \cite{Ma,Ma15}). 
 
 \begin{theo}[Parabolic surgery for the Newton map of a polynomial]
 	\label{Thm:Parabolic_Surgery_Newton}
 	Let $N_p$ be a postcritically finite Newton map of degree $d\ge 3$, and let $\Delta^+_n$ be its marked channel diagram with $1\le n\le d$. For all $1\le j\le n$, let $\cA(\xi_j)$ be the marked basins of superattracting fixed points $\xi_j$. Then there exist a homeomorphism $\phi$ and a postcritically minimal Newton map $N_{\tilde p e^{\tilde q}}$ of degree $d$ with $\deg( \tilde q)=n$ such that:
 	\begin{enumerate}[(a)]
 		\item $\phi\circ N_p(z)=N_{\tilde p e^{\tilde q}}\circ \phi(z)$ for all $z\not \in \bigcup_{1\leq j\leq n}\cA^{\circ}(\xi_j)$; in particular, $\phi:J(N_p)\to J(N_{\tilde p e^{\tilde q}})$ is a homeomorphism which conjugates $N_p$ to $N_{\tilde p e^{\tilde q}}$;
 		\item $\phi(\infty)=\infty$, and $\phi(\bigcup_{1\leq j\leq n}\cA(\xi_j))$ is the full basin of the parabolic fixed point at $\infty$ of $N_{\tilde p e^{\tilde q}}$;
 		\item $\phi$ is conformal on the interior of $\Chat\setminus \bigcup_{1\leq j\leq n}\cA(\xi_j)$;
 		\item the marked invariant accesses of the marked channel diagram $\Delta^+_n$ of $N_p(z)$ correspond to all dynamical accesses of the parabolic basin of $\infty$ for $N_{\tilde p e^{\tilde q}}$.
 	\end{enumerate}
 \end{theo}
 
 Parabolic surgery, as stated above, defines a mapping from the space of Newton maps of polynomials to the space of Newton maps of entire maps that take the form $P(z)\text{exp}(Q(z))$ for polynomials $P(z)$ and $Q(z)$.
 
 The proof of Main Theorem~\ref{Thm:Main} has two parts: injectivity and surjectivity.
 
 Injectivity part is given in Theorem~\ref{Thm:Injectivity_Haissinsky}, which shows that results of parabolic surgeries applied to PCF $N_{p_1}$ and PCF $N_{p_2}$ with markings are affine conjugate if and only if $N_{p_1}$ and $N_{p_2}$ are affine conjugate, and this conjugacy sends the markings of $N_{p_1}$ to the markings of $N_{p_2}$.
 
 Surjectivity part is given in Theorem~\ref{Thm:Surjectivity_Haissinsky}, which states that every PCM Newton map is obtained uniquely from the PCF Newton map of a polynomial and the parabolic surgery. For this, we use Cui's result on parabolic to hyperbolic surgery to perturb the Newton map of $p e^q$ to the Newton map of a polynomial with markings. We then apply parabolic surgery to the latter through its marking and obtain the PCM Newton map of $P e^Q$, for some polynomials $P$ and $Q$. Finally, we show that both PCM Newton map of $p e^q$ that we started with and the PCM Newton map of $P e^Q$ are affine conjugate.
 
 Iterating the Newton map of a polynomial $p$ is referred to as Newton's method for finding the roots of $p$ in the complex plane. It is a classical tool, and in recent studies it was shown that Newton's method is robust and efficient even when the degree of $p$ is over a million; for more progress and details see \cite{SS}.
 In practical applications, adding an exponential factor $e^q$ with a polynomial $q$ comes with a disadvantage. In \cite{Haruta}, Haruta showed that the area of every immediate basin of an attracting fixed point of $N_{pe^q}$ is finite when $\deg(q)\geq 3$. This shows, in particular, that most of the area is taken by basins of $\infty$, where the iterates of the Newton's method applied to $p e^q$ will diverge to $\infty$, thus orbits starting at these points do not lead to roots of $p$.
 
 Preliminary material is presented and proved in \cite{Ma}. For notions used in holomorphic dynamics we refer to \cite{Mdyn}.
 
 \section{Dynamical Properties of Rational Newton maps}
 Let $f:\C\to\C$ be an entire map (polynomial or transcendental entire map). Its Newton map is a meromorphic map given by $N_f(z):=z-f(z)/f'(z)$.
 
 Following \cite{RS}, the Newton map $N_f$ is a rational map if and only if $f=p e^q$ for some polynomials $p$ and $q$. Let $m$, $n\geq 0$ be the degrees of $p$ and $q$, respectively. When $n = 0$ and $m \geq 2$, the point at $\infty$ is repelling with the multiplier $m/(m-1)$. When $n = 0$ and $m = 1$, then the Newton map is constant. If $n \geq 1$, the point at $\infty$ is parabolic with the multiplier $+1$ and multiplicity $n+1 \geq 2$. 
 
 Quadratic Newton maps are trivial. In this paper, we only consider Newton maps of degree at least $3$.
 
 The Julia set of a rational map $f$ is denoted by $J(f)$; its complement is the Fatou set, denoted by $F(f)$. By $\deg(f,z)$, denote the local degree of $f$ at a point $z$, and denote the critical points of $f$ by $C_f = \{z|\deg(f,z)>1 \}$. Denote the postcritical set of $f$ by $P_f=\overline{\bigcup_{n\ge1} f^{\circ n}(C_f)}$. A rational map $f$ is called \emph{postcritically finite} (\emph{PCF}) if $P_f$ is a finite set. It is called \emph{geometrically finite} if the intersection $P_f\cap J(f)$ is a finite set.
 
 The \emph{basin of attraction} of an attracting (parabolic) fixed point $\xi$ of $f$ is defined to be \(\text{int}\{z \in \Chat:\lim_{n\rightarrow \infty}f^{\circ n}(z)=\xi\},\) the interior of the set of starting points $z$ that eventually converge to $\xi$ under iterations of $f$, and is denoted by $\cA(\xi)$. The \emph{immediate basin} of $\xi$, denoted by $\cA^{\circ}(\xi)$, is the forward invariant connected component of the basin $\cA(\xi)$. For parabolic fixed points there could be more than one immediate basin.
 
 \begin{figure}[!h]
 	\centering
 	\includegraphics[scale=.43]{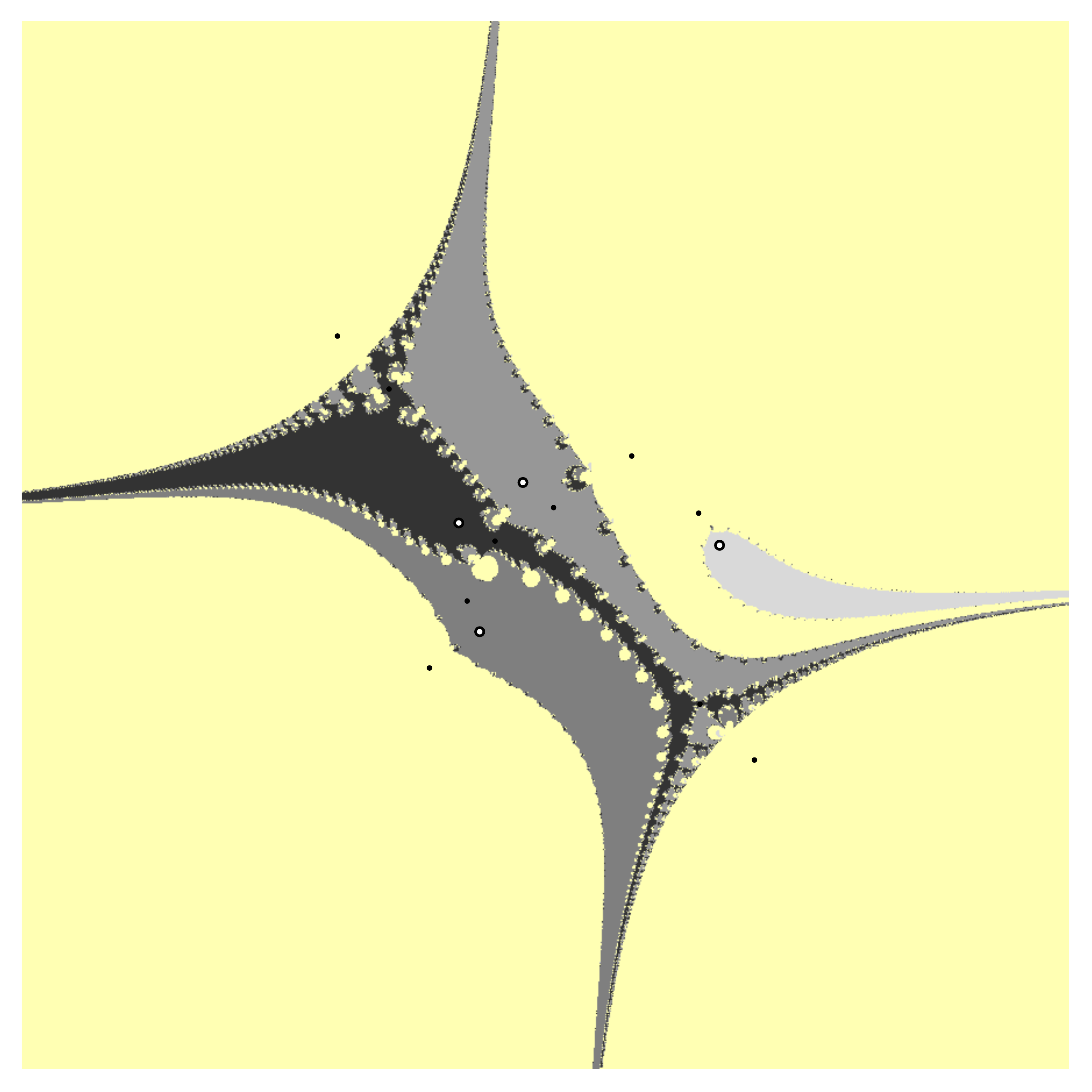}
 	
 	\caption{The Julia set of the Newton map of degree eight with basins is depicted. The basin of a parabolic fixed point at $\infty$ has four petals, one of which has two accesses to $\infty$. White dots with a black circular boundary are fixed points; black dots are non-fixed critical points.}
 	\label{Fig:NewtonBasin}
 \end{figure}
 
 An immediate basin of a fixed point is simply connected and unbounded for rational Newton maps. In \cite{Przytycki}, Przytycki answered a question posed by Manning (see \cite{Man}) and proved that, for Newton maps of polynomials, all immediate basins are simply connected and unbounded. In \cite{MS}, Mayer and Schleicher extended this result to the case of Newton maps of entire maps. Shishikura strengthened these results by proving that \emph{all} components of the Fatou set are simply connected for every rational map with a single weakly repelling fixed point \cite{Shishikura} and that, in particular, the Julia set is connected for all rational Newton maps of entire maps. Generalizing Shishikura's result even further, Bara\'nski {\em et al} in \cite{BFJK14} showed that the Julia set is always connected for all (transcendental) Newton maps of entire maps. The Julia set of a Newton map is
 depicted in Figure \ref{Fig:NewtonBasin}.
 
 \begin{defin}[Invariant access to $\infty$] \label{Def:Access}
 Let $\cA^{\circ}$ be the immediate basin of a fixed point $\xi \in \C$ or the parabolic fixed point
 at $\infty$ of a Newton map $N_{p e^q}$. Fix a point $z_0\in \cA^{\circ}$, and consider a curve $\Gamma:[0,\infty)\to \cA^{\circ}$ with $\Gamma(0) =z_0$ and $\lim_{t\to\infty}\Gamma(t)=\infty$. Its homotopy class (with endpoints fixed) within
 $\cA^{\circ}$ defines an {\em invariant access to $\infty$} in $\cA^{\circ}$.
 \end{defin}
 It is possible to consider any other point $z_0'\in \cA^{\circ}$ as the starting point of a curve $\gamma'$ to $\infty$. If we take a curve $\gamma_0$ joining $z_0$ to $z_0'$, then $\gamma_0\cup \gamma'$ becomes a curve starting at $z_0$ and landing at $\infty$ in the same homotopy class of $\gamma'$, meaning that the choice of $z_0$ is not relevant for the definition of invariant accesses.
 
 Let us consider a curve $\eta$ starting at $z_0$ and landing at $N_{p e^q}(z_0)$. Then both curves $\Gamma:[0,\infty)\to \cA^{\circ}$ starting at $z_0$ and landing at $\infty$ and the curve $N_{p e^q}(z_0)(\Gamma)\cup \eta$ belong to the same invariant access. In \cite{BFJK17}, this type of invariant access to $\infty$ is called {\em strongly invariant access}.
 For Newton maps, there always exists an invariant access called {\em dynamical access} to $\infty$ in immediate basins of the parabolic fixed point at $\infty$. To obtain this access, we consider a curve $\eta$ that joins $z_0$ to $N_{p e^q}(z_0)$ and take the homotopy class of the curve $\Gamma:=\bigcup_{k\ge 0}N_{p e^q}^{\circ k}(\eta)$. The latter lands at $\infty$ and is forward invariant under $N_{p e^q}$. The dynamical accesses are essential in obtaining bijection between the spaces of Newton maps.
 
 Let $N_{pe^{q}}$ be a Newton map, and let $U$ be a component of its Fatou set. A \emph{center} of $U$ is a point $\xi\in U$ such that:
 \begin{enumerate}[(a)]
 	\item when $U$ is a component of a superattracting basin, then $\xi$ is its unique critical periodic point;
 	\item when $U$ is an immediate basin of $\infty$, then $\xi$ is its unique critical point;
 	\item when $U$ is strictly preperiodic, of preperiod $m_U\ge 1$, then $N_{pe^{q}}^{\circ m_U}(\xi)$ is the center of $N_{pe^{q}}^{\circ m_U}(U)$, which is an immediate basin of a superattracting periodic point or of the basin of $\infty$.
 \end{enumerate}
 
 Postcritically minimal Newton maps of entire maps that take the form $pe^q$, for polynomials $p$ and $q$, enjoy similar properties to postcritically finite Newton maps of polynomials.
 In \cite{Ma}, the following result is proved.
 \begin{prop}[Normal forms for PCM Newton maps]
 	\label{Thm:Characterization_PCM}
 	Let $f$ be a PCM Newton map, $U$ be any component of the Fatou set of $f$ and let $V=f(U)$. Then $U$ contains a unique center $\xi_U$. Moreover, there exist Riemann maps $\psi_U:U\to \D$ with $\psi_U(\xi_U)=0$ and $\psi_V:V\to \D$ with $\psi_V(\xi_V)=0$ such that:
 	\begin{enumerate}[(a)]
 		\item if $U$ is an immediate basin of a parabolic fixed point at $\infty$ (in this case, $V=U$), then the following diagram is commutative,
 		\[
 		\begin{diagram}[heads=LaTeX,l>=3em]
 		U&\rTo^{f}   & U\\
 		\dTo^{\psi_U} && \dTo_{\psi_U}\\
 		\D & \rTo^{P_k}& \D
 		\end{diagram}
 		\]
 		where $P_k(z)=(z^k+a)/(1+a z^k)$ with $a=(k-1)/(k+1),$ the parabolic Blaschke product, and $k-1\ge 1$ is the multiplicity of the center of $U$ as a critical point of $f$;
 		\item in all other Fatou components (also including periodic ones), we have the following commutative diagram,
 		\[
 		\begin{diagram}[heads=LaTeX,l>=3em]
 		U&\rTo^{f}& V\\
 		\dTo^{\psi_U}& &\dTo_{\psi_V}\\
 		\D & \rTo^{z\mapsto z^k}&\D
 		\end{diagram}
 		\]
 		where $k-1$ is the multiplicity of the center of $U$ as a critical point of $f$, if the center is not a critical point of $f$, then we let $k=1$.
 	\end{enumerate}
 \end{prop}
 
 Let us define the channel diagram of the postcritically finite Newton map of a \emph{polynomial}. Let superattracting fixed points of a postcritically finite Newton map $N_p$ be denoted by $a_i$ and their immediate basins by $\cA^{\circ}_i$ for all $1\le i\le d$. Let $\phi_i:(\cA^{\circ}_i,a_i) \to (\mathbb{D},0)$ be a
 B\"ottcher coordinate with the property that $\phi_i(N_p(z))=\phi_i^{k_i}(z)$ for each $z \in \D$, where $k_i-1 \geq 1$ is the multiplicity of $a_i$ as a critical point of $N_p$. The power map $z \mapsto z^{k_i}$ fixes $k_i-1$ internal rays in $\mathbb{D}$. Under $\phi_i^{-1}$ these rays map to the $k_i-1$ pairwise disjoint (except for endpoints) simple curves $\Gamma^1_i,\Gamma^2_i,\ldots,\Gamma^{k_i-1}_{i}\subset \cA^{\circ}_i$ that connect $a_i$ to $\infty$ are pairwise non-homotopic in $\cA^{\circ}_i$ and are invariant under $N_p$ as sets. They represent all invariant accesses to $\infty$ in $\cA^{\circ}_i$.
 
 The union
 \[ \Delta = \bigcup_{i=1}^d \bigcup_{j=1}^{k_{i}-1}
 \overline{ \Gamma^{j}_{i} }\]
 forms a connected graph in $\Chat$ that is called the \emph{channel diagram} of $N_p$. It follows that the channel diagram is forward invariant, $N_p(\Delta)= \Delta$. The channel diagram records the mutual locations (embedding) of the immediate basins of $N_p$. Moreover, the channel diagram of a Newton map tells us all about the possible applications of parabolic surgery, but to apply parabolic surgery we only need one access to $\infty$ within an immediate basin. So we need to introduce a marking of the channel diagram.
 
 \begin{defin}[Marked Channel Diagram $\Delta^+_n$] \label{Def:MarkedChanneldiagram}
 For each $i\in \{1,\ldots,d\}$ we mark \emph{at most one} fixed ray $\Gamma^{j^*}_{i}$ in the immediate basin of $a_i$. If a ray in the immediate basin of $a_i$ is marked, then we call the basin of $a_i$ a \emph{marked basin}. The \emph{marked channel diagram} is a channel diagram $\Delta$ with marking, that is an additional information on which fixed rays are selected/marked. If $n\le d$ rays are marked, we denote the $n$-marked channel diagram by $\Delta^+_n$.
 \end{defin}
 A basin can be marked or unmarked. \emph{Marking} defines a single access among all accesses within a \emph{marked} immediate basin through which parabolic surgery will be performed.
 
 Consider a Newton map $N_{p e^q}(z)=z-p(z)/(p'(z)+p(z)q'(z))$ of degree $d\ge3$, and let $\deg (q)=n\le d$, then the number of distinct roots of $p$ is $d-n$. Notice that the leading coefficient of $p$ cancels, so we can assume that $p$ is monic. Similarly, the constant term of $q$ is also not relevant, since we take the derivative of $q$. Any automorphism of $\Chat$ fixing $\infty$ is an affine transformation of the form $z\mapsto a z+b$ ($a\not = 0$), which is, in general, a composition of a scaling and a translation. 
 When $q(z)\not \equiv \text{const.}$, by scaling, we change the leading coefficient of $q$ to any nonzero complex number. For instance, we make $q'$ a monic polynomial. Indeed, a scaling by $a$ conjugates 
 \begin{equation*}
 N_{p e^q}(a z)/a=(a z-\frac{p(a z)}{p'(a z)+p(a z)q'(a z)})/a=z-\frac{p(a z)}{a p'(a z)+p(a z)a q'(a z)}.
 \end{equation*} 
 Let $q'(z)=b_{n-1}z^{n-1}+b_{n-2}z^{n-2}+\cdots$ be the derivative of $q$, where $b_{n-1}\not =0$ is the leading coefficient of $q'(z)$, then we obtain $a q'(a z)=b_{n-1}a^n z^{n-1}+b_{n-2}a^{n-1}z^{n-2}+\cdots$. By a choice of $a$ such that $b_{n-1}a^n=1$, we make $q'(z)$ monic. In other words, if we let $p_a(z):=p(az)$ and $q_z(z):=q(az)$ then $N_{p e^q}(a z)/a=N_{p_a e^{q_a}}(z)$. Now we are only left with one more freedom: essentially, a translation. By translation, we may further assume that either $p$ or $q$ is centered: all roots sum up to zero. Translation by $b$ conjugates
 
 \begin{equation*}
 N_{p e^q}(z+b)-b=z-\frac{p(z+b)}{p'(z+b)+p(z+b)q'(z+b)}.
 \end{equation*}
 
 When $q(z)\equiv \text{const.}$, by translation we make $p$ centered; and by scaling we can have $p(1)=0$. We can change the multiplier of an attracting fixed point of a Newton map by a suitable local quasiconformal surgery, therefore, we may further assume that all roots of $p$ are simple: thus we assume that all finite fixed points are superattracting for Newton maps.
 
 Finally, as explained above, let us normalize polynomials $p$ and $q$ as follows:
 \begin{description}
 	\item if $q\equiv \text{const.}$, we assume that $p$ is centered and $p(1)=0$ (i.e. $z=1$ is a root of $p$);
 	\item if $q\not \equiv \text{const.}$, we assume that $q'$ is monic; moreover, we assume that either $p$ or $q$ (the one with the degree at least $2$) is centered;
 \end{description}
 furthermore, we assume that $p$ is monic and has only simple roots (we achieve this by local surgery).
 
 These lead us to define the following main objects of this paper.
 
 \begin{defin}	\label{Space:Newton}
 For each pair $d\ge 3$ and $1\le n \le d$, denote by $\cN(d-n,n)$ the space of Newton maps $N_{p e^q}$, of degree $d$, normalised as above. For instance, $\cN(d):=\cN(d,0)$ is the space of Newton maps of polynomials $P$, of degree $d$. The polynomials $P$ are monic and centered, they have a root at $z=1$ and all roots are simple.
\end{defin}
 
 \begin{defin}\label{Space:PCM}
 Denote by $\cN_\text{pcf}(d)$ the space of \emph{postcritically finite} Newton maps of polynomials, of degree $d\ge 3$, that are centered, monic and have a root at $z=1$.
\end{defin}
 
 \begin{defin} \label{Space:PCF_Marking}
 For each pair $d\ge 3$ and $1\le n \le d$, denote by $\cN^{+,n}_\text{pcf}(d)$ the space of all postcritically finite Newton maps in $\cN_\text{pcf}(d)$ with all markings $\Delta^+_n$ ($n$-marked channel diagram) at accesses in $n$ marked immediate basins.
\end{defin}
 
\begin{defin}
 \label{Def:SpacePCM}
 For each pair $d\ge 3$ and $1\le n \le d$, denote by $\cN_\text{pcm}(d-n,n)$ the space of \emph{postcritically minimal} Newton maps in $\cN(d-n,n)$.
 \end{defin}
 
 By the above arguments and normalization, we obtain the following lemma.
 \begin{lemma}
 	Assume that $f$ and $\tilde f\in\cN(d-n,n)$ are conjugate by an affine map $\phi$, i.e. $\phi\circ f=\tilde f\circ \phi$.
 	\begin{itemize}
 		\item If $n=0$, the case of the Newton map of a polynomial, then $\phi(z)=z/a$, where $a\neq 0$ is a finite fixed point of $\tilde f$;
 		\item If $n\ge1$, then $\phi(z)=z/a$, where $a^n=1$.
 	\end{itemize}
 \end{lemma}
 
 We don't have a true parameter space; some number of maps are conformally conjugate as can be seen by the above lemma.
 It is now clear that for every $n\le d$ the parameter plane of $\cN(d-n,n)$ is of complex dimension $d-2$.
 
 \section{Plumbing surgery by G. Cui}\label{Ch:Surgery}
 In \cite{C, CT1, CT2},  Cui developed a surgery method, which is called plumbing surgery, to turn parabolic points into hyperbolics: attracting and repelling. Let us state  Cui's result from \cite{C}.
 \begin{theo}[Cui]
 	\label{Thm:Cui}
 	Suppose that $g$ is a geometrically finite rational map and $X$ is a parabolic cycle of $g$. Then there exist a continuous family of geometrically finite sub-hyperbolic rational maps $\{f_t\}$ $\left(1\le t<\infty \right)$ and a continuous family of quasiconformal
 	conjugacies $\{\phi_t\}$ from $f_1$ to $\{f_t\}$, such that the following conditions hold:
 	\begin{enumerate}[(a)]
 		\item	$\{f_t\}$ uniformly converges to $g$ as $t\to\infty$.
 		\item $\{\phi_t\}$ uniformly converges to a quotient map $\phi$ of $\Chat$ as $t\to\infty$ and
 		$\phi\circ f_1=g\circ \phi$, i.e. the following diagram is commutative.
 		\[\begin{diagram}[heads=LaTeX,l>=3em]
 		\Chat &\rTo^{f_1} & \Chat\\
 		\dTo^{\phi} & & \dTo_{\phi}\\
 		\Chat &\rTo^{g} & \Chat.
 		\end{diagram}\]
 		Moreover, $\phi$ is a homeomorphism from $J(f_1)$ onto $J(g)$.
 		\item \label{item-Cui-Thm} For every periodic Fatou domain $D$ of $g$, if $D$ is a parabolic component associated with $X$, then $\phi^{-1}(D)$ is contained in an attracting domain of $f_1$ and $\phi$ is quasiconformal homeomorphism on any domain compactly contained in $\phi^{-1}(D)$.
 		
 		Otherwise, $\phi^{-1}(D)$ is a Fatou domain of $f_1$ and $\phi$ is conformal on $\phi^{-1}(D)$.
 	\end{enumerate}
 \end{theo}
 
 The theorem uses the following notion.
 
\begin{defin}[Quotient map]
 \label{Def:QuotientMap}
 Let $h$ be a continuous surjective map on $\Chat$. The map $h$ is called a \emph{quotient} map if $h^{-1}(y)$ is either a singleton or a full continuum for every point $y\in\Chat$, i.e. $\Chat\setminus h^{-1}(y)$ is a simply connected domain.
\end{defin}
 
 \begin{xrem}
 Note that $f_1$ in the above theorem is a sub-hyperbolic geometrically finite map: all of its non-repelling cycles are attracting. The theorem converts all parabolic domains into attracting domains. Since the semi-conjugacy $\phi$ is conformal in other Fatou components, the multipliers of attracting cycles of $g$ are preserved. For a postcritically minimal Newton map $g$, item $(\ref{item-Cui-Thm})$ of the theorem allows us to conclude that $f_1$ could be further changed to a postcritically finite Newton map by a standard quasiconformal surgery.
\end{xrem}
 We use the following lemma during the construction of a local topological conjugacy between Newton maps at their parabolic fixed points at $\infty$, for its proof please refer to \cite[Lemma 3.4.]{CT2}
 \begin{lemma}
 	\label{Lem:Partial_Local}
 	Suppose rational maps $f$ and $g$ with parabolic fixed points $z_0$ and $z_1$, respectively, are given. Let $\phi_0$ be a $K$-quasiconformal conjugacy between their attracting flowers. Then for any $\epsilon >0$, there is a local $(K+\epsilon)$-quasiconformal conjugacy $\phi$ between $f$ and $g$ such that $\phi=\phi_0$ on a smaller attracting flower.
 \end{lemma}
 
 We shall use the following fact on extensions of quasisymmetric maps between the boundaries of quasidiscs and quasiannuli.
 \begin{prop}[Quasiconformal interpolation]\cite[Proposition~2.30]{BF14}\label{Prop:Interpolation}
 	\begin{enumerate}[(a)]
 		\item Suppose $G_1$ and $G_2$ are quasidiscs bounded by $\gamma_1$ and $\gamma_2$ respectively, and let $f:\gamma_1\to\gamma_2$ be quasisymmetric. Then $f$ extends to a quasiconformal map $\hat{f}:\overline{G_1} \to \overline{G_2}$.
 		\item For $j=1, 2$, suppose $A_j$ are open quasiannuli bounded by the quasicircles $\gamma_j^i,\gamma_j^o$. Let $f^i: \gamma_1^i\to\gamma_2^i$ and $f^o: \gamma_1^o\to\gamma_2^o$ be quasisymmetric maps between the inner and outer boundaries respectively. Then there exists a quasiconformal map $f:\overline{A_1} \to \overline{A_2}$ extending the boundary maps $f^i$ and $f^o$.
 	\end{enumerate}
 \end{prop}
 
 The following is a classical result on lifting properties of covers that we use later.
 \begin{lemma}
 	\label{Lem:Lifting}
 	Let $Y,Z$ and $W$ be path-connected and locally path-connected Hausdorff spaces with base points $y\in Y, z\in Z$ and $w\in W$. Suppose $p:W\to Y$ is an unbranched covering map and $f:Z\to Y$ is a continuous map such that $f(z)=p(w)=y$.
 	\[\begin{diagram}[heads=LaTeX,l>=3em]
 	Z,z &\rTo^{\tilde f} & W,w\\
 	& \rdTo_{f} & \dTo_{p}\\
 	& & Y,y
 	\end{diagram}\]
 	There exists a unique continuous lift $\tilde f$ of $f$ to $p$ with $\tilde f (z)=w$ for which the above diagram is commutative i.e. $f=p\circ \tilde f$ if and only if the induced homomorphisms $f_*: \pi_1(Z, z) \to \pi_1(Y, y)$ and $p_*:\pi_1(W, w)\to(Y, y)$ at the level of fundamental groups satisfy \[ f_*(\pi_1(Z,z)) \subset p_* (\pi_1(W,w)), \] where $\pi_1$ denotes the fundamental group.
 \end{lemma}
 
 \section{Injectivity of parabolic surgery}
 
 Parabolic surgery defines a mapping from the space of $n$-marked postcritically finite Newton maps of polynomials (recall that it is denoted by $\cN^{+,n}_\text{pcf}(d)$) to the space of postcritically minimal Newton maps of entire maps that take the form $pe^q$ with $\deg(q)=n$ (denoted by $\cN_\text{pcm}(d-n,n)$). Different surgeries applied to the same Newton map of a polynomial with its different accesses may produce rational maps that are affine conjugate. For the simplest case: $n=1$, we have two ways of applying parabolic surgery to $2z^3/(3z^2-1)\in \cN_\text{pcf}(3)$ along its (two distinct) immediate basins with single accesses to $\infty$ in each (see Fig.~\ref{Fig:CubicNewtonDouble} for its Julia set~\footnote{The Newton map $2z^3/(3z^2-1)$ is conjugate via $z\to i/(\sqrt{2}z)$ to the cubic polynomial $z^3+\frac32z$.}). The resulting Newton maps of these parabolic surgeries will have a parabolic basin with a single immediate basin. There exists a single PCM Newton map with that property in $\cN_\text{pcm}(2,1)$. It is $z-(z^2+c)/(z^2+2z+c)$ for $c=-\frac{1}{4}$ (see Fig.~\ref{Fig:Parabolic} Left for its Julia set~\footnote{The Newton map $z-(z^2+c)/(z^2+2z+c)$ for $c=-\frac{1}{4}$ is conjugate via $z\to i/z-1/2$ to the cubic polynomial $z^3-iz^2+z$.}). Thus both applications of parabolic surgery produce the same result.
 \begin{figure}[!ht]
 	\centering
 	\includegraphics[scale=.21]{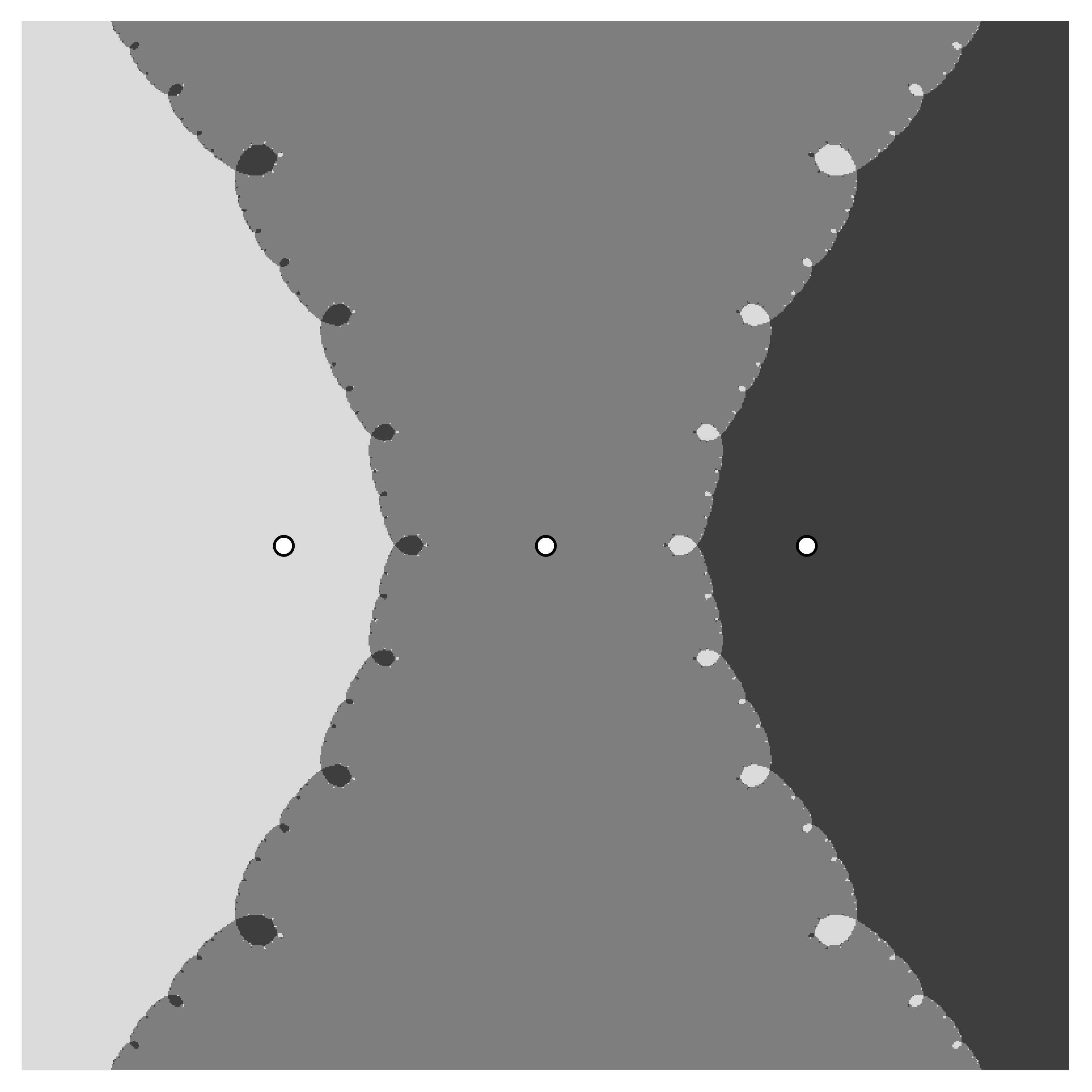}
 	
 	\caption{The Julia set of $2 z^3/(3 z^2-1)$, the cubic Newton map, where $z=0$ is a superattracting fixed point with full invariant basin in gray area, the other superatracting basins are in light gray and dark gray areas respectively.}
 	\label{Fig:CubicNewtonDouble}
 \end{figure}
 \begin{figure}[!ht]
 	\centering
 \includegraphics[scale=.15]{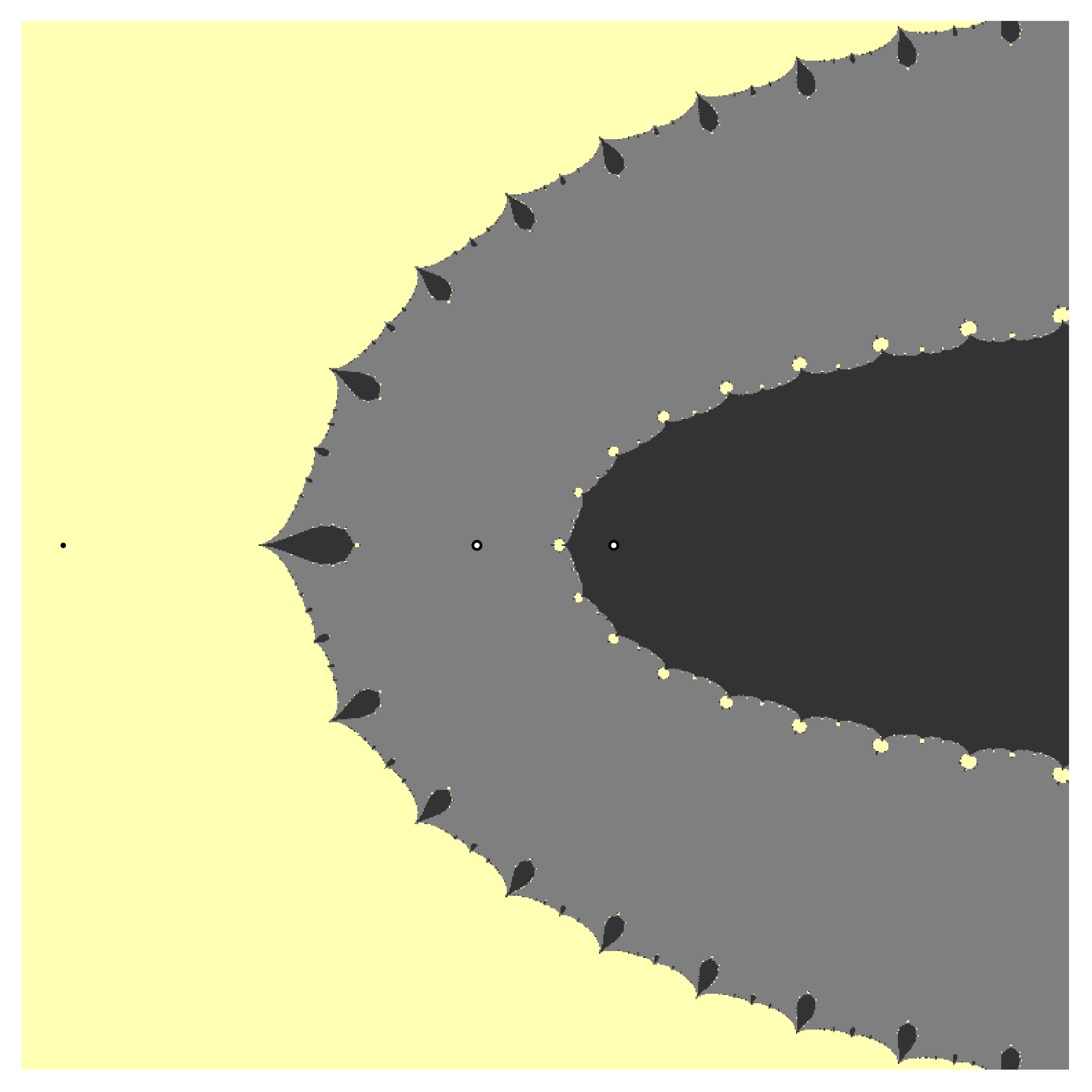}
 x\hspace{5mm}
 \includegraphics[scale=.167]{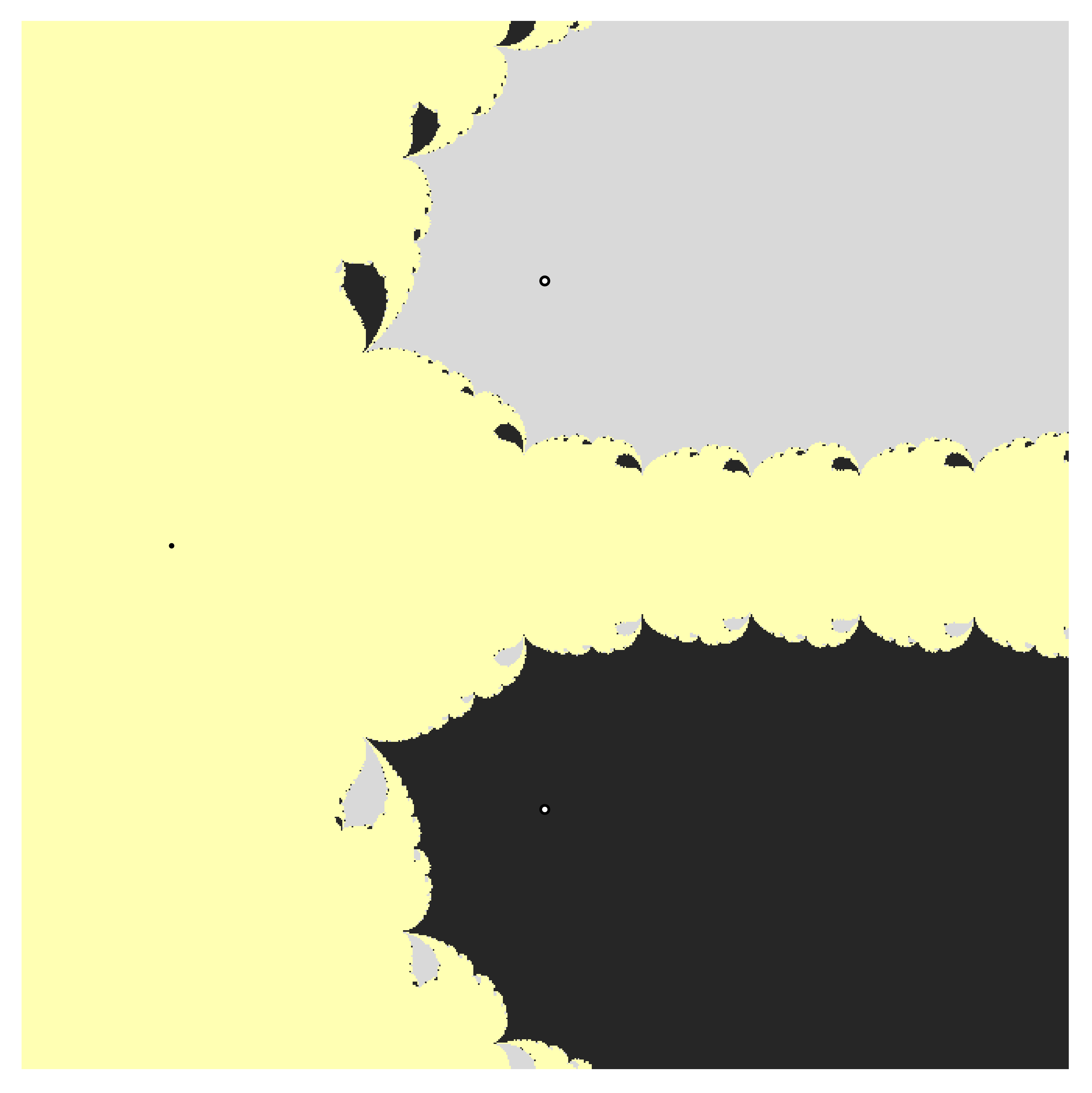}
 	
 	\caption{{\em Left:} The Julia set of $z-(z^2+c)/(z^2+2z+c)$ for $c=-\frac{1}{4}$. The basin of fixed point $z=-0.5$ is in gray area and the basin of other fixed point is in dark gray area, the basin of $\infty$ is in light yellow area on the left. {\em Right:} The Julia set of $z-(z^2+c)/(z^2+2z+c)$ for $c=2$. The basins of fixed points are in light gray and dark gray areas, the basin of $\infty$ has two accesses and is in light yellow area, respectively.}
 	\label{Fig:Parabolic}
 \end{figure} 
 
 Similarly, consider applications of parabolic surgery to $2z^3/(3z^2-1)$ through its third immediate basin, which has two distinct accesses. We can perform parabolic surgery in two ways, but results are same. It is easy to see that the result is $z-(z^2+c)/(z^2+2z+c)$ for $c=2$, which is a unique PCM Newton map with two accesses in the immediate basin of $\infty$ (see Fig.~\ref{Fig:Parabolic} Right for its Julia set).
 We identify this kind of ``distinct" parabolic surgeries if their results are same or M\"obius conjugate to each other. It is easy to see that the relation under this identification is an equivalence relation. Let us state it in the following as a definition.
 
\begin{defin} [$\sim_H$ Ha\"issinsky equivalence] \label{Def:Surgery_Equivalence}
 Let $F$ and $G$ be results of applications of parabolic surgery to $N_{p_1}$ with marking $\Delta^+_n(N_{p_1})$ and $N_{p_2}$ with marking $\Delta^{\prime +}_n(N_{p_2})$, both belonging to $\cN^{+,n}_\text{pcf}(d)$, respectively. The two parabolic surgeries are said to be Ha\"issinsky equivalent if there exists an affine map $M$ such that $M\circ F=G\circ M$. Notation $\sim_H$ is used for Ha\"issinsky equivalent surgeries.
\end{defin}
 
 The following theorem characterizes equivalent parabolic surgeries, which states that distinct surgeries produce non-conjugate (distinct) rational maps unless the underlying maps with markings themselves are conjugate and the conjugacy interchanges the markings.
 
 \begin{theo}[Injectivity of parabolic surgery]
 	\label{Thm:Injectivity_Haissinsky}
 	For every pair of natural numbers $d\ge3$ and $1\le n\le d$, parabolic surgeries applied to $N_{p_1}$ with its $n$-marking $\Delta^+_n(N_{p_1})$ and $N_{p_2}$ with its $n$-marking $\Delta^{\prime +}_n(N_{p_2})$, both belonging to $\cN^{+,n}_\text{pcf}(d)$, are Ha\"issinsky equivalent if and only if there exists an affine map $L$ such that
 	\begin{itemize}
 		\item $L\circ N_{p_1}=N_{p_2}\circ L$; and
 		\item $L(\Delta^+_n(N_{p_1}))=\Delta^{\prime +}_n(N_{p_2}).$
 	\end{itemize}
 	In other words, the mapping $\cF_n: \cN^{+,n}_\text{pcf}(d)/\sim_H \to \cN_{\text{pcm}}(d-n,n)$ induced by parabolic surgery is an injective mapping, which preserves embedding and the dynamics of Julia sets.
 \end{theo}
 
 \begin{proof}
 For one direction, assume we have an affine map $L$ such that:
 \begin{itemize}
 	\item $L\circ N_{p_1}=N_{p_2}\circ L$
 	\item $L(\Delta^+_n(N_{p_1}))=\Delta^{\prime +}_n(N_{p_2})$.
 \end{itemize} 
 Let us apply parabolic surgery to $N_{p_1}$ and $N_{p_2}$ through marked channel diagrams $\Delta^+_n(N_{p_1})$ and $\Delta^{\prime +}_n(N_{p_2})$ respectively, then the result trivially follows by the construction of parabolic surgery. The converse is the main part of the theorem, which we deal with now.
 
 For the other direction, let us use simpler notations: $f:=N_{p_1}$, $g:=N_{p_2}$, and let $F$ and $G$ be the resulting functions of parabolic surgery applied to $f$ with marking $\Delta^+_n(f)$ and $g$ with marking $\Delta^{\prime +}_n(g)$ respectively. 
 For $1\leq j\leq n$, let us denote by $\cA(\xi_j)$ the marked basins of $f$. By Theorem~\ref{Thm:Parabolic_Surgery_Newton}, there exists a homeomorphism $\phi_f:\Chat\to\Chat$ such that the following diagram commutes.
 \[
 \begin{diagram}[heads=LaTeX,l>=3em]
 \Chat\setminus \cup_{1\leq j\leq n}\cA^{\circ}(\xi_j) &\rTo^{f} & \Chat\setminus \cup_{1\leq j\leq n}\cA^{\circ}(\xi_j) \\
 \dTo^{\phi_f} &      & \dTo_{\phi_f}\\
 \Chat\setminus \phi_f\bigl(\cup_{1\leq j\leq n}\cA^{\circ}(\xi_j)\bigl) & \rTo^{F} & \Chat\setminus \phi_f\bigl(\cup_{1\leq j\leq n}\cA^{\circ}(\xi_j)\bigl) 
 \end{diagram}
 \eqno \text{D}1
 \]
 where $\Chat\setminus \cup_{1\leq j\leq n}\cA^{\circ}(\xi_j)$ is the complement of the union of marked \emph{immediate basins} of $f$.
 Moreover, $\cA_{F}(\infty)=\phi_f(\cup_{1\leq j\leq n}\cA(\xi_j))$, and it is the parabolic basin of $\infty$ for $F$.
 As above, for $1\leq j\leq n$, let us denote by $\cA(\chi_j)$ basins of marked superattracting fixed points $\chi_j$ of $g$.
 
 Similarly, by Theorem~\ref{Thm:Parabolic_Surgery_Newton}, there exists a homeomorphism $\phi_g$ such that the following diagram commutes.
 \[
 \begin{diagram}[heads=LaTeX,l>=3em]
 \Chat\setminus \cup_{1\leq j\leq n}\cA^{\circ}(\chi_j) &\rTo^{g} & \Chat\setminus \cup_{1\leq j\leq n}\cA^{\circ}(\chi_j) \\
 \dTo^{\phi_g} & & \dTo_{\phi_g}\\
 \Chat\setminus \phi_g\bigl(\cup_{1\leq j\leq n}\cA^{\circ}(\chi_j)\bigl) &\rTo^{G} & \Chat\setminus \phi_g\bigl(\cup_{1\leq j\leq n}\cA^{\circ}(\chi_j)\bigl) 
 \end{diagram}
 \eqno \text{D}2
 \]
 where $\Chat\setminus \cup_{1\leq j\leq n}\cA^{\circ}(\chi_j)$ is the complement of the union of marked \emph{immediate basins} of $g$.
 Moreover, $\cA_{G}(\infty)=\phi_g(\cup_{1\leq j\leq n}\cA(\chi_j))$ is the parabolic basin of $\infty$ for $G$. 
 
 Assume both surgeries are equivalent: $F\sim_H G$, i.e. there exists an affine map $M$ with the following commutative diagram.
 \[
 \begin{diagram}[heads=LaTeX,l>=3em]
 \Chat &\rTo^{F} & \Chat\\
 \dTo^{M} & & \dTo_{M}\\
 \Chat &\rTo^{G} & \Chat
 \end{diagram}
 \eqno \text{D}3
 \]
 By D3, we obtain $M(\cA_{F}(\infty))=\cA_{G}(\infty)$ and $M(\Chat\setminus \cA_{F}(\infty))=\Chat\setminus \cA_{G}(\infty)$. Moreover, the dynamical accesses of $\cA_{F}(\infty)$ for $F$ transform via $M$ to the dynamical accesses of $\cA_{G}(\infty)$ for $G$.
 From diagrams D1, D2, and D3, it follows that \[\phi^{-1}_g\circ M\circ \phi_f \circ f=g\circ \phi^{-1}_g\circ M \circ\phi_f\] on $\Chat\setminus \cup_{1\leq j\leq n}\cA(\xi_j)$. 
 The homeomorphism $$\psi^1=\phi^{-1}_g\circ M\circ \phi_f:\Chat\setminus \cup_{1\leq j\leq n}\cA^{\circ}(\xi_j)\to \Chat\setminus \cup_{1\leq j\leq n}\cA^{\circ}(\chi_j)$$ conjugates $f$ to $g$ in $\Chat\setminus \cup_{1\leq j\leq n}\cA^{\circ}(\xi_j)$, the complement of the union of marked immediate basins of $f$.
 
 We want to extend $\psi^1$ to $\mathbb{S}^2$ (topological 2-sphere) as a homeomorphism that is also a topological conjugacy between $f$ and $g$, and what is missing are the marked immediate basins $\cup_{1\leq j\leq n}\cA^{\circ}(\xi_j)$ of $f$. 
 To accomplish this, we use normalized Riemann maps (B\"ottcher coordinates) coming from Proposition \ref{Thm:Characterization_PCM}. 
 Let us sort the indices such that $\cA^{\circ}(\xi_j)$ and their counterparts $\cA^{\circ}(\chi_j)$ are cyclically ordered at $\infty$ for $1\leq j\leq n$. For every $1\leq j\leq n$, by Proposition~\ref{Thm:Characterization_PCM} there exists a Riemann map 
 $\psi_{j_f}:(\cA^{\circ}(\xi_j),\xi_j)\to (\D,0)$ such that $\psi_{j_f}\circ f \circ \psi^{-1}_{j_f}(z)=z^{k_j}$, where $k_j=\deg(f,\xi_j)$. We have $k_j-1$ choices for $\psi_{j_f}$. 
 
 Let $R(t)=\{r e^{2\pi i t}: 0\leq r\leq 1\}$ be the \emph{radial line at angle} $t$ in $\D$. We fix some choice of a Riemann map $\psi_{j_f}$ and define $R_j(t)=\psi^{-1}_{j_f}(R(t))$, \emph{the ray of angle} $t$ in $\cA(\xi_j)$. 
 The radial lines $R(t)$ at angles $t\in\{0,\frac{1}{k_j-1},\ldots,\frac{k_j-2}{k_j-1}\}$ are fixed by $z\mapsto z^{k_j}$. Hence, the rays in $\cA^{\circ}(\xi_j)$ at those angles are fixed by $f$, and the rays define all invariant accesses to $\infty$ within the immediate basin. Once we label each access, the different choices of Riemann maps $\psi_{j_f}$ do nothing but cyclically permute (a shift) the labels of accesses. Note that accesses do not depend on the choice of a Riemann map. Let us choose the Riemann map $\psi_{j_f}$ such that the rays at angles $0,\frac{1}{k_j-1},\ldots,\frac{k_j-2}{k_j-1}$ in $\cA^{\circ}(\xi_j)$ are ordered in a counter-clockwise direction, and the $0$ ray being the one which is marked. By \cite{TY}, the Julia set of $f$ is locally connected as $f$ is geometrically finite and its Julia set is connected, so the boundary of every Fatou component of $f$ is locally connected, hence, every ray lands by Carath\'eodory's theorem. Also note that every $f$-invariant ray lands at $\infty\in \partial \cA^{\circ}(\xi_j)$.
 
 We have the corresponding construction for $g$ plane: the Riemann maps 
 \[\phi_{j_g}(\cA^{\circ}(\chi_j),\chi_j)\to (\D,0)\] such that $\psi_{j_g}\circ g \circ \psi^{-1}_{j_g}(z)=z^{k_j}$,
 where $k_j=\deg(f,\xi_j)=\deg(g,\chi_j)$ for all $1\le j\le n$. We normalize these Riemann maps of marked immediate basins of $g$ in the same ordering used for $f$. We define the rays in immediate basins $\cA^{\circ}(\chi_j)$ in $g$ plane; as above, every ray lands.	 
 
 We construct conjugating maps between corresponding marked immediate basins of $f$ and $g$. For every $1\le j\le n$, consider the map \[\psi^2_j:=\phi^{-1}_{j_g}\circ \phi_{j_f}:\cA^{\circ}(\xi_j)\to \cA^{\circ}(\chi_j),\] which is conformal, and it conjugates $f$ to $g$ on its domain of definition as the following diagrams commute.
 \[\begin{diagram}[heads=LaTeX,l>=3.5em]
 \cA^{\circ}(\xi_j)  	&\rTo^{f} & \cA^{\circ}(\xi_j) \\
 \dTo^{\psi_{j_f}}		& 		& \dTo_{\psi_{j_f}}\\
 \D 		 	&\rTo^{z\mapsto z^{k_j}} 	& \D\\
 \uTo^{\psi_{j_g}}		& 			& \uTo_{\psi_{j_g}}\\
 \cA^{\circ}(\chi_j) 	&\rTo^{g} & \cA^{\circ}(\chi_j)
 \end{diagram}\]
 
 It is now natural to check if both $\psi^1$ and $\psi^2_j$ match up on $\partial \cA^{\circ}(\xi_j)$ for all $1\le j\le n$.
 For all $1\le j\le n$, we define an equivalence relation on $\S^1$, which denotes the unit circle, for $\psi_{j_f}$ (and $\psi_{j_g}$), where the equivalence classes of rays (identified by angles) contain those rays landing at a common point. Alternatively, since the inverse of $\psi_{j_f}$ (correspondingly the inverse of $\psi_{j_g}$) has a continuous extension to the closed unit disk by Carath\'eodory's theorem, every equivalence class consists of points of $\S^1$ that are mapped to the same point under the continuous extension of the inverse of $\psi_{j_f}$ (correspondingly the continuous extension of the inverse of $\psi_{j_g}$). 
 
 All $f$-invariant rays land at $\infty$, and thus these belong to the same class. All iterated pre-fixed rays (the iterated image is an invariant ray) split into distinct equivalent classes.
 It is clear that our equivalence relation is generated by the closure of the equivalence relation defined by $f$-invariant rays and their iterated preimages. By the normalized Riemann maps, $\psi_{j_f}$ for $f$, and $\psi_{j_g}$ for $g$, we obtain the same equivalence relation for both $f$ and $g$.
 Indeed, the map $\psi^1$ sends bijectively the iterated preimages of $\infty$ in the $f$ plane to the corresponding iterated preimages of $\infty$ in the $g$ plane. Thus $\psi^2_j$ extends continuously to the closure $\overline{\cA^{\circ}(\xi_j)}$. Since $\infty\in J(f)$ therefore iterated preimages of $\infty$ are dense in $\partial\cA^{\circ}(\xi_j)$, hence for every point $z\in \partial\cA^{\circ}(\xi_j)$ the equivalence class of rays landing at $z$ is the limit of classes of rays landing at iterated preimages of $\infty$ in $\partial\cA^{\circ}(\xi_j)$. Moreover, the extension (denote again by $\psi^2_j$) coincides with $\psi^1$ on the iterated preimages of $\infty$. By construction the maps $\psi^1$ and $\psi^2_j$ agree on a dense subset of their common domains of definition; namely, on the point at $\infty$ and its iterated preimages in $\partial\cA^{\circ}(\xi_j)$. It follows that $\psi^1$ and $\psi^2_j$ coincide everywhere on their common domains of definition. Hence the orientation preserving homeomorphism 
 \[\psi = \left \{ 
 \begin{array}{ll}
 \psi^1(z), & z\in \mathbb{S}^2\setminus \cup_{1\leq j\leq n}\cA^{\circ}(\xi_j),\\
 \psi^2_j(z), & z\in \cA^{\circ}(\xi_j), \ \text{for}\ {1\leq j\leq n}
 \end{array} \right. \]
 conjugates $f$ to $g$ in $\mathbb{S}^2$.
 
 Finally, we invoke the rigidity part of Thurston's theorem on the characterization of branched coverings \cite{DH2} (in fact we need to apply a result from \cite{BCT} where we add some extra marked points to the postcritical set, for us, it is a point at $\infty$) to degree $d\geq 3$ functions $f$ and $g$ deducing the existence of $L$, a conformal conjugacy $L\circ f=g\circ L$\footnote{Alternatively, by the proof structure of Chapter 6 of \cite{DH1} as well as of the next section of this paper, surjectivity of parabolic surgery, we can construct the conformal conjugacy by hand by keeping the conformal conjugacy at small disc neighborhoods of superattracting periodic points of $f$ compactly contained in their immediate basins and interpolating this conformal map to a quasiconformal map $\phi_0$ of the sphere. Next, we keep taking lifts and obtain a sequence of quasiconformal maps with the same complex dilatation. We only need to require for all $m>0$, $\phi_m\circ f=g\circ \phi_{m+1}$ and $\phi_m=\phi_0$ in a small disc neighborhood of some superattracting fixed point of $f$ so that we fix a base point from this domain to define the lifts. The sequence $\{\phi_m\}_{m\ge 0}$ has a convergent subsequence. Let $\phi$ be its limit. It is clear that $\phi$ is a conformal map of $\Chat$ since the domain where $\phi_m$ are not conformal shrinks to the Julia set of $f$. The claim follows since the Julia set of $f$ has measure zero. In the domain we have $\phi=\phi_0$. We have constructed the initial map to satisfy $\phi_0\circ f=g\circ \phi_0$ in the domain, hence $\phi\circ f=g\circ \phi$ by the identity principle of holomorphic functions.}. Moreover, $L$ sends the marked fixed points of $\Delta^+_n(f)$ to those of $\Delta^{\prime +}_n(g)$, hence all of the marked channel diagram: $L(\Delta^+_n(f))=\Delta^{\prime +}_n(g)$. 
\end{proof}
 
 \section{Surjectivity of parabolic surgery}
 
 For a PCM Newton map with the parabolic fixed point at $\infty$, G. Cui's plumbing surgery perturbes its parabolic domains into attracting domains, thus producing a sub-hyperbolic rational map, which then is turned to the postcritically finite Newton map of a polynomial with its marked accesses to $\infty$. The latter is done by the standard surgery: changing multipliers at the attracting fixed points and in preimage components of it if there are critical points. Then, by parabolic surgery, we do the reverse of this process: for the obtained postcritically finite Newton map of a polynomial, we change its repelling fixed point at $\infty$ into a parabolic fixed point, thus obtaining a rational map, which turns out to be a PCM Newton map. We show that the latter is affine conjugate to the postcritically minimal Newton map we started with.
 \begin{theo}[Surjectivity of parabolic surgery]
 	\label{Thm:Surjectivity_Haissinsky}	
 	For every pair of natural numbers $d\ge3$ and $1\le n\le d$, parabolic surgery induces a (natural) surjective mapping $\cF_n$ from the quotient space $\cN^{+,n}_\text{pcf}(d)/\sim_H$ onto the space of affine conjugacy classes of PCM Newton maps in $\cN_\text{pcm}(d-n,n)$.
 \end{theo}
\begin{proof}
 The proof is involved. \emph{Idea of proof.} We apply Cui's plumbing surgery to a PCM Newton map from $\cN_\text{pcm}(d-n,n)$ to perturb its parabolic fixed point at $\infty$. The resulting rational map is then converted to the postcritically finite Newton map of a polynomial in $\cN^{+,n}_\text{pcf}(d)$ via intermediate surgery. We then apply parabolic surgery to the last Newton map of a polynomial to produce a PCM Newton map in $\cN_\text{pcm}(d-n,n)$. We show that the PCM Newton map we took from $\cN_\text{pcm}(d-n,n)$ and the result of the parabolic surgery are affine conjugate to each other. Thus, proving that parabolic surgery induces a surjective mapping from the space $\cN^{+,n}_\text{pcf}(d)/\sim_H$ to the space of affine conjugacy classes in $\cN_\text{pcm}(d-n,n)$. The proof is split into four major parts, Part A-Part D. For each part, let us provide more details of the proof idea in the following.
 
  {\em Part A - Perturbation of parabolic fixed point}. Apply Cui plumbing surgery (Theorem~\ref{Thm:Cui}) to a PCM Newton map $N_{p_1 e^{q_1}}\in \cN_\text{pcm}(d-n,n)$ to obtain a rational map $f_1$ and a quotient map $\phi$ such that $\phi\circ f_1=N_{p_1 e^{q_1}}\circ \phi$. The injectivity of $\phi$ is broken only in Fatou components of $f_1$ that map to parabolic domains of $N_{p_1 e^{q_1}}$, in particular, $\phi$ is a homeomorphism when restricted to $J(f_1)$. Next, we change $f_1$ in its attracting basins such that the result of this intermediate surgery produces a postcritically finite Newton map, denote it by $N_p$. We have a choice for $f_1$ but all choices produce the same $N_p$, thus we obtain a unique and canonical mapping.
 	
 {\em Part B - Parabolic surgery application}. We apply parabolic surgery to $N_p$ of Part A, with its corresponding marked channel diagram, which is uniquely obtained from $N_{p_1 e^{q_1}}$. Denote by $N_{p_2 e^{q_2}}$ the result of parabolic surgery. 
 	
 {\em Part C - Construction of topological conjugacy}. We construct a \emph{topological} conjugacy $\Psi$ between $N_{p_1 e^{q_1}}$ and $N_{p_2 e^{q_2}}$ by cutting parabolic basins where the conjugacy is broken and gluing the conjugacy coming from the normalized Riemann maps. This part is only needed to make sure that we have correct choices of Riemann maps in parabolic basins, as these are not unique. Alternatively, it is also possible to skip this part and make this choice during the next part, where we again cut those domains to obtain a global conformal conjugacy.
 	
 {\em Part D - Construction of conformal conjugacy}. Using the topological conjugacy of Part C, which is locally conformal on the Fatou set of $N_{p_1 e^{q_1}}$, and giving up the conjugacy we had, we construct a sequence of quasiconformal homeomorphisms that are conformal conjugacies between the Newton maps at petals of parabolic fixed point and in neighborhoods of superattracting basins. Every element of the sequence is the lift of the previous element, and the domain of conjugacy increases, eventually filling the whole Fatou set. Finally, by extracting a convergent subsequence, we obtain a \emph{conformal} conjugacy between $N_{p_1 e^{q_1}}$ and $N_{p_2 e^{q_2}}$.

 \vspace*{.1cm}
 
 \subsection*{Part A - Perturbation of parabolic fixed point.} Fix $d\ge 3$ and $1 \le n \le d$. Let $N_{p_1 e^{q_1}}\in \cN_\text{pcm}(d-n,n)$ be a postcritically minimal Newton map. We invoke Cui plumbing surgery (Theorem~\ref{Thm:Cui}) to deduce a sub-hyperbolic rational map $f_1$ and a quotient map $\phi$ such that $\phi\circ f_1=N_{p_1 e^{q_1}}\circ \phi$, i.e. the following diagram is commutative.
 \[ \begin{diagram}[heads=LaTeX,l>=3em]
 \Chat        &\rTo^{f_1}   & \Chat\\
 \dTo^{\phi} &         & \dTo_{\phi}\\
 \Chat        &\rTo^{N_{p_1 e^{q_1}}}   & \Chat
 \end{diagram}\]
 Moreover, when restricted to $J(f_1)$, $\phi$ is a homeomorphism from $J(f_1)$ onto $J(N_{p_1 e^{q_1}})$. Now we study essential properties of $f_1$ and $\phi$. Without loss of generality, we can assume that $\infty$ is a non-attracting fixed point of $f_1$, after M\"obius conjugation if necessary. Thus, $\infty\in J(f_1)$. Then we obtain $\phi(\infty)= \infty $ since $\phi(\infty)=N_{p_1 e^{q_1}}(\phi(\infty))$ and $\infty$ is the only fixed point of $N_{p_1 e^{q_1}}$ on its Julia set.
 For the Newton map $N_{p_1 e^{q_1}}$ its parabolic cycle consists only of a point at $\infty$. For every immediate basin $U$ of $\infty$, the map $\phi$ can be obtained as a quasiconformal map on any domain compactly contained in $\phi^{-1}(U)$, then $\phi^{-1}$ sends a critical point of $N_{p_1 e^{q_1}}$ in $U$ to a critical point of $f_1$ in $\phi^{-1}(U)$. Let $c\in U$ be the center of $U$, a unique critical point of $N_{p_1 e^{q_1}}$ in $U$. Since $\phi$ is a homeomorphism restricted to the Julia set, we have $\deg(f_1,\phi^{-1}(c))=\deg(N_{p_1 e^{q_1}},c)$, thus there is no other critical point of $f_1$ in $\phi^{-1}(U)$. Indeed, let $K$ be a neighborhood of $\phi^{-1}(c)$ compactly contained in $\phi^{-1}(U)$; by the theorem, we choose $\phi$ such that it is quasiconformal on $K$, thus $\phi^{-1}(c)$ is a single point, moreover, it is a critical point of $f_1$.
 The following diagram commutes,
 \[\begin{diagram}[heads=LaTeX,l>=3em]
 f_1^{-1}(K)        &\rTo^{f_1}   & K\\
 \dTo^{\phi} &      & \dTo_{\phi}\\
 \phi(f_1^{-1}(K)) &\rTo^{N_{p_1 e^{q_1}}}   & \phi(K)
 \end{diagram}\]
 hence $\phi$ is quasiconformal on $f_1^{-1}(K)$. Induction shows that $\phi$ is quasiconformal in all of the iterated preimages of $K$.
 Now assume $c_1$ is a critical point of $N_{p_1 e^{q_1}}$ such that $N_{p_1 e^{q_1}}^{\circ l}(c_1)=c\in U$ for the minimal $l\ge1$, i.e. $l$ is the preperiod of the component containing $c_1$. Since $\phi$ is a homeomorphism with the above commutative diagram, it follows that after iteratively applying the conjugacy for iterative preimages of $K$, we obtain that $\phi^{-1}(c_1)$ is a critical point of $f_1$, and $f_1^{\circ l}(\phi^{-1}(c_1))=\phi^{-1}(c)$ for the same minimal $l\ge 1$. Moreover, since $\phi$ is a homeomorphism on the Julia set we have $\deg(f_1,\phi^{-1}(c_1))=\deg(N_{p_1 e^{q_1}},c_1)$. Furthermore, $\phi^{-1}(c_1)$ is the only critical points of $f_1$ in the Fatou component containing $\phi^{-1}(c_1)$. 
 
 Similarly, by induction we shall show that $\phi$ is conformal in every $\phi^{-1}(U)$, where $U$ is a Fatou component of $N_{p_1 e^{q_1}}$ that is not a parabolic domain. These types of components could only be components of basins of superattracting periodic points (including fixed) of $N_{p_1 e^{q_1}}$. If $U$ is a superattracting immediate basin of $N_{p_1 e^{q_1}}$ then by Cui plumbing theorem (Theorem~\ref{Thm:Cui}) $\phi^{-1}(U)$ is an immediate basin of a superattracting periodic point of $f_1$ and $\phi$ is conformal in $\phi^{-1}(U)$, therefore $\phi^{-1}$ sends superattracting periodic points of $N_{p_1 e^{q_1}}$ to those of $f_1$. Let $V$ be a component of $N_{p_1 e^{q_1}}^{-1}(U)$ other than $U$. We have the following commutative diagram,
 
 \[\begin{diagram}[heads=LaTeX,l>=3em]
 \phi^{-1}(V)     &\rTo^{f_1}  	 & \phi^{-1}(U)\\
 \dTo^{\phi} &        		 & \dTo_{\phi}\\
 V     &\rTo^{N_{p_1 e^{q_1}}}		 & U
 \end{diagram}\]
 hence $\phi$ is conformal in $\phi^{-1}(V)$. By induction, $\phi$ is conformal in $\phi^{-1}\circ N_{p_1 e^{q_1}}^{\circ(- k)}(U)$ for all $k\ge 1$. What we have that, for every component of $F(f_1)$ that is preserved by the conjugacy $\phi$, the critical orbits terminate in finite time.
 
 We have to mention that, in all immediate basins of $f_1$ that are counterparts to the parabolic domains of $N_{p_1 e^{q_1}}$, we can change the multipliers to zero (see \cite[Chapter~4.2]{BF} and \cite[Theorem~5.1]{CG}; compare with \cite[Lemma 3.8]{Ma}). By carefully checking the process of the latter, we can achieve that there is a single grant orbit in that basin, then the resulting function is a postcritically finite Newton map, denote it by $N_p$. In this process, we have that the new rational map$N_p$ and the old $f_1$ are conjugate except in small neighborhoods of attracting fixed points of $f_1$. This intermediate surgery produces a quasiconformal homeomorphism $\phi_1$ such that the following diagram is commutative,
 \[ \begin{diagram}[heads=LaTeX,eqno,l>=3em]
 \Chat\setminus \phi_1^{-1}(A)    &\rTo^{N_p}   & \Chat\setminus \phi_1^{-1}(A)\\
 \dTo^{\phi_1}					  &        	    & \dTo_{\phi_1}\\
 \Chat\setminus A       		  &\rTo^{f_1}   & \Chat\setminus A
 \end{diagram}
 \]
 where $A$ is the union of all basins that are \emph{not} affected by the intermediate surgery. Moreover, $\phi_1$ is conformal in the interior of $\Chat\setminus \phi_1^{-1}(A)$. Let us summarize what we have obtained so far. 
 \begin{itemize}
 	\item The quotient map $\phi$, when restricted to the Julia set of $f_1$ is a topological conjugacy between the Julia sets of $f_1$ and $N_{p_1 e^{q_1}}$. Moreover, $\phi$ is conformal (conjugacy) on the Fatou components of $f_1$. These Fatou components are counterpart to the non-parabolic Fatou domains of $N_{p_1 e^{q_1}}$.
 	\item The quasiconformal homeomorphism $\phi_1$ is a conjugacy between $f_1$ and $N_p$ on the complement of the union of disk neighborhoods of (Cui surgery created) attracting fixed points of $f_1$ and it is conformal in the rest of the basins, including all basins of superattracting periodic points of $f_1$. Thus, $\phi\circ\phi_1$, a quotient map, is a topological conjugacy between the Julia sets of $N_p$ and $N_{p_1 e^{q_1}}$, and it is a conformal map where $\phi$ is conformal.
 \end{itemize}
 Normalize $N_p$ to make the polynomial $p$ monic, centered, and having a root at $z=1$, so that it belongs to $\cN_{\text{pcf}}(d)$. We \emph{mark} the basins of $N_p$ that are created by Cui plumbing surgery. We also need \emph{marked} accesses in every marked basin. By \cite[Proposition 2.3]{Ma} (see also \cite[Corollary C]{BFJK17}), every parabolic immediate basin of $N_{p_1 e^{q_1}}$ has its unique dynamical access. Note that since $\phi^{-1}$ restricted to the Julia set of $N_{p_1 e^{q_1}}$ homeomorphically sends boundaries of its parabolic basins to boundaries of attracting basins of $f_1$, all (dynamical) accesses of former transform to all (marked) accesses of the latter via $\phi$ (and further via $\phi^{-1}_1$ to $N_p$). Thus, we have marked accesses of $N_p$ in its corresponding marked basins that are counterparts to parabolic basins of $N_{p_1 e^{q_1}}$.
 
 \subsection*{Part B - Parabolic surgery application.}  For ${1\leq j\leq n}$, let us denote by $\cA(\xi_j)$ marked basins. We also have marked access in each of $\cA(\xi_j)$. We apply parabolic surgery (Theorem~\ref{Thm:Parabolic_Surgery_Newton}) to $N_p$ through those marked basins and accesses deducing a homeomorphism $\phi_2$ and a postcritically \emph{minimal} Newton map $N_{p_2 e^{q_2}}$ such that:
 \begin{itemize}
 	\item $\phi_2$ is conformal in every Fatou component of $N_p$ that is not marked,
 	\item $\phi_2\circ N_p(z) =N_{p_2 e^{q_2}}\circ \phi_2(z)$ for all $z\notin \bigcup_{1\leq j\leq n}\cA^{\circ}(\xi_j)$ i.e. the following diagram commutes.
 	\[	\begin{diagram}[heads=LaTeX,l>=3.5em]
 	\Chat\setminus \cup_{1\leq j\leq n}\cA^{\circ}(\xi_j)  		 &\rTo^{N_p}   				& \Chat\setminus \cup_{1\leq j\leq n}\cA^{\circ}(\xi_j)\\
 	\dTo^{\phi_2}												 &        	  				& \dTo_{\phi_2}\\
 	\Chat\setminus\phi_2(\cup_{1\leq j\leq n}\cA^{\circ}(\xi_j)) &\rTo^{N_{p_2 e^{q_2}}}    & \Chat\setminus \phi_2(\cup_{1\leq j\leq n}\cA^{\circ}(\xi_j))
 	\end{diagram}\]
 \end{itemize}
 
 \subsection*{Part C - Construction of topological conjugacy.} We shall construct a topological conjugacy between $N_{p_1 e^{q_1}}$ and $N_{p_2 e^{q_2}}$ that is conformal in the Fatou set of $N_{p_1 e^{q_1}}$. By above constructions it follows that the map $\Psi=\phi_2\circ\phi_1^{-1}\circ \phi^{-1}$ is a conjugacy between $N_{p_1 e^{q_1}}$ and $N_{p_2 e^{q_2}}$ on the complement of parabolic basins of $N_{p_1 e^{q_1}}$. Moreover, $\Psi$ is conformal in the basins of superattracting periodic points (including fixed) of $N_{p_1 e^{q_1}}$. We want to extend this conjugacy to the parabolic basin $\cA^1(\infty)$ as well. We construct our topological conjugacy by gluing the Riemann maps of corresponding parabolic components.
 
 For every ${1\leq j\leq n}$, let $\cA_j^{\circ1}$ be an immediate basin of the parabolic fixed point of $N_{p_1 e^{q_1}}$, and let $c_j^1$ be the center of $\cA_j^{\circ1}$. For every ${1\leq j\leq n}$, since $\Psi(\partial\cA_j^{\circ1})$ is $N_{p_2 e^{q_2}}$-invariant, it is the boundary of exactly one parabolic component of $N_{p_2 e^{q_2}}$, denote it by $\cA_j^{\circ2}$. Let $c_j^2$ be the only critical point in $\cA_j^{\circ2}$, the center of $\cA_j^{\circ2}$. Let $\psi_j^{\circ1}:\cA_j^{\circ1}\to\D$ and $\psi_j^2:\cA_j^{\circ2}\to\D$ be the corresponding uniquely defined Riemann maps sending the critical points $c_j^1$ and $c_j^2$ to the origin and the fixed point at $\infty$ to $z=1$ (as the parabolic points are accessible), moreover, we have $k_j=\deg(N_{p_1 e^{q_1}},c_j^1)=\deg(N_{p_2 e^{q_2}},c_j^2)$.\\ The following diagrams commute,
 \[	\begin{diagram}[heads=LaTeX,l>=3.5em]
 \cA_j^{\circ1}    	&\rTo^{N_{p_1 e^{q_1}}} & \cA_j^{\circ1} \\
 \dTo^{\psi_j^1}		&        	   			& \dTo_{\psi_j^1}\\
 \D       		 	&\rTo^{P_{k_j}}   		& \D\\
 \uTo^{\psi_j^2}		&        	   			& \uTo_{\psi_j^2}\\
 \cA_j^{\circ2}   	&\rTo^{N_{p_2 e^{q_2}}} & \cA_j^{\circ2}
 \end{diagram}\]
 where $P_{k_j}(z)=(z^{k_j}+a_j)/(1+a_j z^{k_j})$ with $a_j=(k_j-1)/(k_j+1)$ is the degree $k_j$ parabolic Blaschke product of $\D$. Note that under these normalizations, the marked access for both immediate basins are mapped via the Riemann maps to the same dynamical access associated to the invariant ray $(0,1)$ for $P_{k_j}$. For every ${1\leq j\leq n}$, the composition $\psi_j^2 \circ(\psi_j^1)^{-1}:\cA_j^{\circ1}\to\cA_j^{\circ2}$ is a conformal conjugacy on $\cA_j^{\circ1}$ between $N_{p_1 e^{q_1}}$ and $N_{p_2 e^{q_2}}$. 
 
 By Carath\'eodory's theorem, the inverses of both maps $\psi_j^1$ and $\psi_j^2$ extend to the boundary of the unit disk. We define an equivalence relation on the unit circle $\S^1$ induced by the extension: $x\sim y\in \S^1$ if and only if both are mapped to the same point on the boundary by the inverse of $\psi_j^1$. Similarly, we define an equivalence relation for the inverse map of $\psi_j^2$ in another copy of the unit circle $\S^1$. We shall show that these two maps define the same equivalence relation on $\S^1$. Indeed, we have $k_j+1$ fixed points of $P_{k_j}$, of which $k_j-2$ are distinct repelling fixed points, and a triple fixed point (a double parabolic) at $z=1$. In total there are $k_j-1$ invariant accesses, all of which correspond to invariant accesses to $\infty$ both immediate basins $\cA_j^{\circ1}$ and $\cA_j^{\circ2}$. By definition of the equivalence relation, we identify all fixed points $P_{k_j}$ since they all map to $\infty$ under the inverse map. Now we take preimages of a given fixed point of $P_{k_j}$. Similarly, in $\cA_j^{\circ1}$ we take the preimages of $\infty$ by $N_{p_1 e^{q_1}}$. The Newton map is locally injective away from its critical points, the invariant rays (accesses) to $\infty$ have preimages which land at the poles in $\partial\cA_j^{\circ1}$, one for each (non-homotopic rays) accesses to $\infty$. This is transported by the Riemann map $\psi_j^1$ to the unit disk and we identify preimages of fixed points according to the rules as in $\cA_j^{\circ1}$. This gives us $k_j-1$ distinct classes of identifications on $\S^1$, one for each corresponding pole other than $\infty$ of $N_{p_1 e^{q_1}}$ in $\partial\cA_j^{\circ1}$. Continuing this process, we split iterated preimages of all fixed points of $P_{k_j}$ in $\S^1$ into equivalence classes coming from iterated preimages of $z=1$ that correspond to the iterated preimages of $\infty$ on $\partial\cA_j^{\circ1}$. We take the closure of this equivalence relation. Since the above diagram commutes, the closed equivalent relations for $\psi_j^1$ and $\psi_j^2$ are the same. 
 
 Thus, the map $\psi_j^2 \circ(\psi_j^1)^{-1}:\cA_j^{\circ1}\to\cA_j^{\circ2}$ extends to the boundary as a continuous map and equals to $\Psi$ on a dense set of points on the common domain of definition, namely on $\infty$ and its iterated preimages. Denote the continuous extension by $\Psi_j^2$. Note that $\Psi_j^2=\Psi$ on $\partial\cA_j^{\circ1}$. The conjugacy is now extended to all of immediate basins of the parabolic fixed point at $\infty$. 
 
 Now we extend it to all other components of the parabolic basin $\cA^1(\infty)$. For a given ${1\leq j\leq n}$, let $U$ be a component of $N_{p_1 e^{q_1}}^{-1}(\cA_j^{\circ1})$ other than $\cA_j^{\circ1}$. Let $c_u$ be the center of $U$, which maps to the critical point in $\cA_j^{\circ1}$, and let $k=\deg(N_{p_1 e^{q_1}},c_u)$ be the local degree. Then $\Psi(\partial U)$ is the boundary of a unique component of $N_{p_2 e^{q_2}}(\cA_j^{\circ2})$, denote the component by $V$, and let $c_v$ denote its unique center. There exist Riemann maps $\psi_U:U\to\D$ and $\psi_V:V\to\D$ such that $\psi_U(c_u)=\psi_V(c_v)=0$ with the following commutative diagrams.
 \[\begin{diagram}[heads=LaTeX,l>=3.5em]
 U			&\rTo^{N_{p_1 e^{q_1}}} & \cA_j^{\circ1} \\
 \dTo^{\psi_U}	&        	   			& \dTo_{\psi_j^1}\\
 \D       		&\rTo^{z\mapsto z^k}   	& \D\\
 \uTo^{\psi_V}	&        	   			& \uTo_{\psi_j^2}\\
 V   			&\rTo^{N_{p_2 e^{q_2}}} & \cA_j^{\circ2}
 \end{diagram}\]	
 These Riemann maps are unique up to post-composition by a rotation of $k^\text{th}$ root of unity. Since we are interested in the composition ${\psi_V}^{-1}\circ\psi_U$, the choice of Riemann maps for both can be restricted to one. Let us fix any choice for $\psi_U$. Now we choose the map $\psi_V$ to be compatible with the dynamics of the Newton maps. Observe that preimages of the invariant ray (the marked access which is associated to the interval $(0,1)$, the zero ray) by $N_{p_1 e^{q_1}}$ in $U$ are mapped by $\psi_U$ to the preimages under $z\mapsto z^k$ of the invariant rays landing at fixed points for $P_{k_j}$ (e.g. $(0,1)$), since the above diagrams are commutative. 
 Note that the map $z\mapsto z^k$ cannot differentiate between these distinct preimages. The map $\Psi$ that is a homeomorphism from $\partial U$ onto $\partial V$ comes in handy. Once $\psi_U$ is chosen, we fix $\psi_V$ in such a way that those preimages of $(0,1)$ by $z\mapsto z^k$ are pulled back to $U$ such that they land at the corresponding points dictated by $\Psi$. There is only one choice of $\psi_V$ for doing this. This choice is compatible with the dynamics of both $N_{p_1 e^{q_1}}$ and $N_{p_2 e^{q_2}}$ on the boundaries of their corresponding Fatou components. 
 
 We define equivalence relations on two copies of $S^1$ for $\psi_U$ and $\psi_V$ correspondingly as we did above. These equivalence relations are the same since both agree on a dense set of common points. Hence ${\psi_V}^{-1}\circ\psi_U$ extends to the closure of $U$ and coincides with $\Psi$ on a dense set of points, thus both are equal on the common domain of definition. This way we extend $\Psi$ to all (first level) components of preimages of immediate parabolic basins. 
 
 We can continue in the same way to extend $\Psi$ to the full basin of $\infty$, since the dynamics are conjugate to the same power maps $z\mapsto z^k$, where $k$ is the common local degree of Newton maps at centers of corresponding components of basins. 
 
 Let us summarize what we have proved so far and recall the definition of $\Psi$, for which we spent the whole Part C,
 \[\Psi= \left \{ 
 \begin{array}{ll}
 \phi_2\circ\phi_1^{-1}\circ \phi^{-1}, & \text{on } \mathbb{S}^2\setminus\cA^1(\infty),\\
 {\psi_V}^{-1}\circ\psi_U, & \text{on }  U \\
 \end{array} \right.\]
 where $U$ and $V$ run over all of the respective components of $\cA^1(\infty)$ and $\cA^2(\infty)$, and all the involved maps in the definition of $\Psi$ are defined. Thus, $\Psi$ is a conjugacy between Newton maps $N_{p_1 e^{q_1}}$ and $N_{p_2 e^{q_2}}$ on $\mathbb{S}^2$, it is conformal in every component of the Fatou set of $N_{p_1 e^{q_1}}$. 
 We still have to show that $\Psi$ is globally continuous on $\mathbb{S}^2$.
 \begin{xcm}
 The map $\Psi$ defined in Part C is a homeomorphism of $\mathbb{S}^2$.
 \end{xcm}
 
 \begin{proof}[Proof of Claim.]
 It suffices to prove continuity of $\Psi$ on $J(N_{p_1 e^{q_1}})$.
 Let us fix $\epsilon>0$, and a sequence of positive reals $\epsilon_k\to 0$ as $k\to\infty$. Consider a sequence of points $\{w_s\}_{s\geq 1}\subset\mathbb{S}^2$ such that $w_s\to w\in J(N_{p_1 e^{q_1}})$ as $s\to \infty$. We shall prove that 
 \begin{equation}\label{Eq:limit}
 \Psi(w_s)\to \Psi(w)\ \text{as}\ s\to\infty.
 \end{equation}
 Since $\Psi$ is a homeomorphism on the Julia set of $N_{p_1 e^{q_1}}$ if we have an inclusion $\{w_{s_k}\}_{k\geq 1}\subset J(N_{p_1 e^{q_1}})$ for some subsequence $\{w_{s_k}\}_{k\geq 1}$ of $w_s$, then $\{\Psi(w_{s_k})\}_{k\geq 1}\subset J(N_{p_1 e^{q_1}})$, hence, the limit \eqref{Eq:limit} holds over the subsequence $\{w_{s_k}\}_{k\geq 1}$. Moreover, if some subsequence $\{w_{s_k}\}_{k\geq 1}$ is contained in $U_{N_{p_1 e^{q_1}}}$, a component of $F(N_{p_1 e^{q_1}})$ (i.e. $\{w_{s_k}\}_{k\geq 1}\subset U_{N_{p_1 e^{q_1}}}$), then along this subsequence the limit \eqref{Eq:limit} holds true since the restriction $\Psi |_{\overline{U}}$ is continuous. Therefore, without loss of generality, we assume that $\{w_s\}_{s\geq 1}\subset F(N_{p_1 e^{q_1}})$ and no subsequence of $\{w_s\}_{s\geq 1}$ is completely contained in a single component of the Fatou set $F(N_{p_1 e^{q_1}})$. As a result of this assumption, the sequence $\{w_s\}_{s\geq 1}$ leaves any given component of $F(N_{p_1 e^{q_1}})$ in finite time. Local connectivity of the Julia set implies that there are only finitely many components of $F(N_{p_1 e^{q_1}})$ with spherical diameter bigger than any given $\epsilon>0$. Now we fix $k$. Sooner or later, points of $\{w_s\}_{s\geq 1}$ leave any Fatou component of $N_{p_1 e^{q_1}}$ with spherical diameter $\geq\epsilon_k$. Observe that the spherical distance between $\psi(w_s)$ and $\psi(w'_s)$ is less than $\epsilon_k$ for all large enough $s$, where $w'_s$ is any point on the boundary of the component where $w_s$ is located, in particular, $w'_s$ is located on $J(N_{p_1 e^{q_1}})$. Clearly along the same ideas, $w'_s\to w$ as $s\to\infty$. Then $w'_s$ converges to the same $w$, since $\Psi$ is continuous on $J(N_{p_1 e^{q_1}})$. The claim is now proved. 
\end{proof}

 \subsection*{Part D - Construction of conformal conjugacy}\footnote{Note that the topological conjugacy $\Psi$ of Part C in this setup is not a c-equivalence between $N_{p_1 e^{q_1}}$ and $N_{p_2 e^{q_2}}$ according to the generalization of Thurston's topological characterization of postcritically finite covering maps to the setting of geometrically finite covering maps with parabolic cycles (please refer to \cite{CT2} for the theory).}
 Using the topological conjugacy of Part C, which is conformal at petals and superattracting domains of $N_{p_1 e^{q_1}}$, by applying an interpolation technique several times we construct the set of quasiconformal homeomorphisms of $\Chat$, which is denoted by $\{\Psi_1,\ldots \Psi_k\}$, where $k$ is the total number of superattracting periodic points of $N_{p_1 e^{q_1}}$. 
 
 Next, we work with $\Psi_k$ from the previous step, and construct, by taking lifts of the local conjugation, $\{\psi_m\}_{m\geq 0}$ a sequence of quasiconformal homeomorphisms of $\Chat$ with uniformly bounded complex dilatation. Finally, by extracting a convergent sub-sequence of the latter we obtain a \emph{conformal} conjugacy between $N_{p_1 e^{q_1}}$ and $N_{p_2 e^{q_2}}$ finishing the proof of the theorem. We divide the dynamical plane of $N_{p_1 e^{q_1}}$ into two parts: some Jordan neighborhood of infinity and the complement of it.
 
 We use $\Psi$ as an initial partial conjugacy between petals at $\infty$ of $N_{p_1 e^{q_1}}$ and $N_{p_2 e^{q_2}}$. Note that, in particular, $\Psi$ is a conformal conjugacy restricted on immediate basins of $\infty$. Let us fix an $\epsilon=1$ (the exact value of $\epsilon$ is not relevant) and a flower at $\infty$. Since $\Psi$ is conformal in the flower, thus $1$-quasiconformal homeomorphism, by Lemma~\ref{Lem:Partial_Local} we obtain a $2$-quasiconformal homeomorphism $\phi$ defined locally at $\infty$ that is a conjugacy between $N_{p_1 e^{q_1}}$ and $N_{p_2 e^{q_2}}$ such that $\phi=\Psi$ on a smaller attracting flower bounded by curves $l_1,\ldots,l_n$. Extend the conjugacy to big petals to include critical points, denote the extended conjugacy and petal boundaries with previously used notations. The extension is possible thanks to conformality of $\Psi$, thus it still conjugates the Newton maps. We fix some quasicircle $L_1$ in the domain of definition of $\phi$ such that $L_1$ separates all superattracting periodic points (including fixed points) of $N_{p_1 e^{q_1}}$ from the big flower, including critical points and their orbits, where we had the equality $\phi=\Psi$. 
 Denote by $L_1^+$ and $L_1^-$ the unbounded and bounded components of the complement of $L_1$ respectively. Consider $L_2=\phi(L_1)$ the corresponding quasicircle in the dynamical plane of $N_{p_2 e^{q_2}}$. Moreover, $L_2$ separates the attracting flower from all superattracting periodic points of $N_{p_2 e^{q_2}}$. Similarly, denote by $L_2^+$ the unbounded component of the complement of $L_2$, and by $L_2^-$ the bounded component. 
 
 We shall extend $\phi$ to the bounded domain $L_1^-$ as a quasiconformal homeomorphism that is conformal on disk neighborhoods of superattracting cycles of $N_{p_1 e^{q_1}}$. Moreover, it will be equal to the map $\Psi$ defined above, thus a conformal conjugacy between $N_{p_1 e^{q_1}}$ and $N_{p_2 e^{q_2}}$ on neighborhoods of superattracting periodic points (including fixed).
 
 In case there exists no superattracting periodic (including fixed) points of $N_{p_1 e^{q_1}}$ we extend $\phi$ by applying item~(a) of Proposition~\ref{Prop:Interpolation} to $L^-_1$. 
 In case there exist superattracting periodic points or fixed points of $N_{p_1 e^{q_1}}$, we extend $\phi$ by applying item~(b) of Proposition~\ref{Prop:Interpolation} to $L^-_1$ sequentially in small disks about all periodic points specified below. Let $C_1,C_2,\ldots,C_k$ denote the list of disjoint simple closed analytic curves contained in $L^-_1$, one for each element of superattracting cycles (the critical point and its orbit) that bound an element in its immediate basin for $N_{p_1 e^{q_1}}$ (see Figure~\ref{Fig:multiannulus} for the illustration of the construction of interpolations). For every $i\le k$, let $\Omega^1_i$ be the closed disk bounded by $C_i$, then $\Omega^1_i\Subset L^-_1$. 
 
 \begin{figure}[ht]
 	\centering
 	\includegraphics[scale=.32]{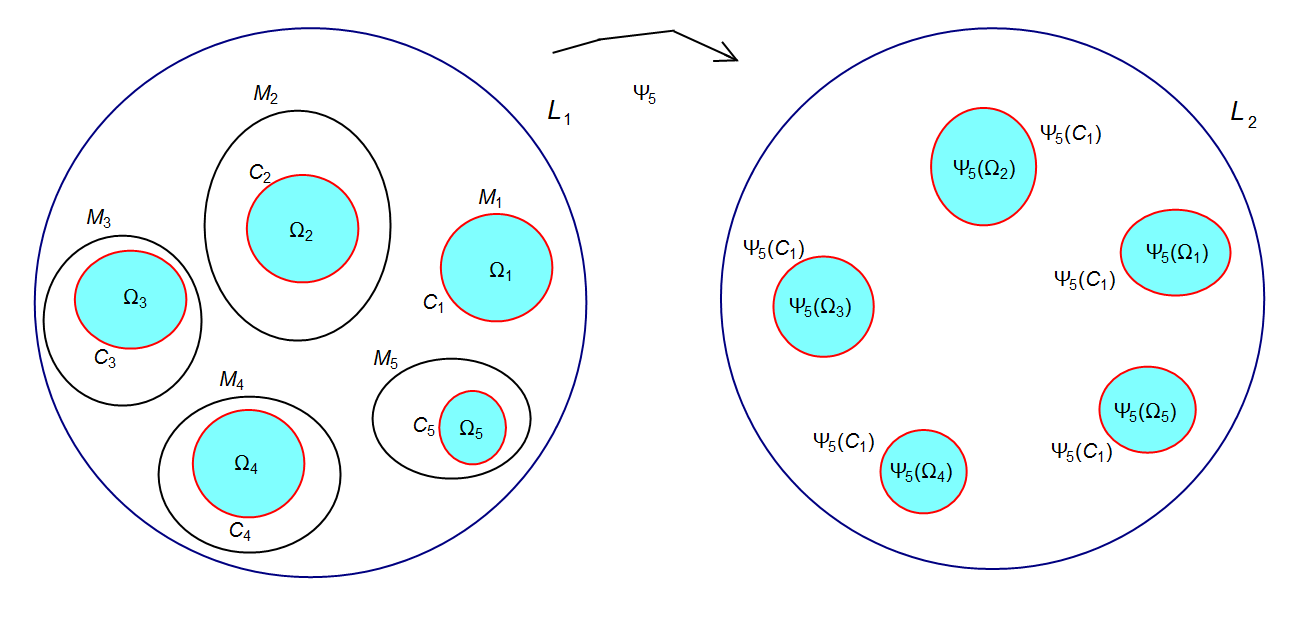}
 	
 	\caption{A schematic illustration of the construction of interpolations. {\em Left:} The analytic disks $\Omega^1_i$ in $N_{p_1 e^{q_1}}$ plane are shown. {\em Right:} The corresponding image for $N_{p_2 e^{q_2}}$ plane.}
 	\label{Fig:multiannulus}
 \end{figure}	
 
 Observe that the images $\Psi(\Omega^1_1),\Psi(\Omega^1_2),\ldots,\Psi(\Omega^1_k)$ are closed disks in $L_2^-$ bounded by analytic curves $\Psi(C_1),\Psi(C_2),\ldots,\Psi(C_k)$, each of which surrounds the corresponding superattracting periodic point (or fixed point) of $N_{p_2 e^{q_2}}$ in its immediate basin. 
 
 We are in position to apply item~(b) of Proposition~\ref{Prop:Interpolation}. First, consider a quasiannulus with the internal boundary $C_1$ and the external boundary $L^1$. Interpolate inner and outer maps $\Psi|_{C_1}$ and $\Psi|_{L_1}$ using item (b) of Theorem~\ref{Prop:Interpolation} to produce a quasiconformal homeomorphism of $\Chat$, denote it by $\Psi_1$. 
 
 Second, we continue the application of item~(b) of Proposition~\ref{Prop:Interpolation} to the next analytic curve $C_2$ and the map $\Psi_1$, which is obtained in the first step. We need to specify the boundary maps, too. There exist many ways to do this. One way is to shrink the curve $C_2$, while keeping the center unchanged, which is a superattracting periodic point (or fixed) of $N_{p_1 e^{q_1}}$. Taking an analytic curve $\tilde{C_2}$ within $\Omega^1_2$ suffices. Another way is to take some quasicircle located within $L^-_1$, denote it by $M_2$, which bounds the curve $C_2$ and separates it from $C_1$. Following the latter way, we have $\Psi_1|_{M_2}$ and $\Psi_1|_{C_2}$ as external and internal maps respectively. Interpolation gives us a quasiconformal homeomorphism of $\Chat$, denote this map by $\Psi_2$. Moreover, $\Psi_2$ is conformal on the union of $\Omega_1^1$ and $\Omega_2^1$. 
 
 Finally, we take some quasicircle located within $L^-_1$, denote it by $M_k$, that bounds the curve $C_k$ and separates it from all other curves $C_1,C_2,\ldots,C_{k-1}$. We consider $\Psi_{k-1}|_{M_k}$ and $\Psi_{k-1}|_{C_k}$ as external and internal maps respectively for the next interpolation. We obtain a quasiconformal homeomorphism of $\Chat$, denote it by $\Psi_k$. Moreover, $\Psi_k$ is conformal on all of $\Omega_i^1$ for $i\le k$.
 
 To ease the notations, let us denote the last interpolating map by $\psi_0$, i.e. $\psi_0:=\Psi_k$, and denote by $\Omega^1$ the union of $\Omega^1_i$ for $1\leq i\leq k$ and the open parabolic flower bounded by $l_1,\ldots,l_n$, and let $\psi_0(\Omega^1)=\Omega^2$. By construction $\psi_0=\Psi$ and $\psi_0\circ N_{p_1 e^{q_1}}=N_{p_2 e^{q_2}}\circ\psi_0$ on $\Omega^1$. i.e. the following diagram commutes.
 \[
 \begin{diagram}[heads=LaTeX,l>=3em]
 \Omega^1        &\rTo^{N_{p_1 e^{q_1}}}   & N_{p_1 e^{q_1}}(\Omega^1) \\
 \dTo^{\phi_0} &      & \dTo_{\phi_0}\\
 \Omega^2 & \rTo^{N_{p_2 e^{q_2}}}   & N_{p_2 e^{q_2}}(\Omega^1)
 \end{diagram}
 \]
 \subsection*{Lifting to obtain a sequence of quasiconformal maps.} 
 Recall that $C_f$ denotes the set of critical points of $f$. Let us define sets: for $i\in\{1,2\}$, $V^i:=N_{p_i e^{q_i}}(C_{N_{p_i e^{q_i}}})$, the set of critical values of $N_{p_i e^{q_i}}$, and $T^i:=N_{p_i e^{q_i}}^{-1}(V^i)$, the preimage of $V^i$ under $N_{p_i e^{q_i}}$. If $\phi_0(V^1)\not=V^2$ then we include the sets $\phi_0^{-1}(V^2)$ and $\phi_0^{-1}(V^1)$ to $V^1$ and $V^2$ respectively. Then we define corresponding $T^i$'s.
 
 For $i\in\{1,2\}$, we have unbranched covering maps $N_{p_i e^{q_i}}:\Chat\setminus T^i\to\Chat\setminus V^i$. The maps $\psi_0\circ N_{p_1 e^{q_1}}:\Chat\setminus T^1\to\Chat\setminus V^2$ and $\psi^{-1}_0\circ N_{p_2 e^{q_2}}:\Chat\setminus T^2\to\Chat\setminus V^1$ are unbranched covering maps as well. Let $\Omega^0$ be a component of $\Omega^1$. Our Newton map has at least one petal at $\infty$, so as $\Omega^0$ we take a petal of $\infty$, a component of $\Omega^1$ assotiated to a petal at $\infty$. Let us fix a base point $x_0\in\Omega^0\setminus O(C_{N_{p_2 e^{q_2}}})$ for the domain $\Chat\setminus V^2$, where $O(C_{N_{p_2 e^{q_2}}})$ denotes the union of grand orbits of critical points of $N_{p_2 e^{q_2}}$. We have $V^2\cup T^2\subset O(C_{N_{p_2 e^{q_2}}})$. In fact, more is true: for $i\in\{1,2\}$, the grand total orbit of critical points $O(C_{N_{p_i e^{q_i}}})$ is generated by $V^i$ (and also by $T^i$). Since $\Omega^0\subset N_{p_2 e^{q_2}}^{-1}(\Omega^0)$, observe that $N_{p_2 e^{q_2}}^{-1}(\Omega^0)$ has many components in the immediate basin associated to the petal $\Omega^0$. As a base point for the domain $\Chat\setminus T^2$, let us fix a preimage $N_{p_2 e^{q_2}}^{-1}(x_0)$ in $N_{p_2 e^{q_2}}^{-1}(\Omega^0)$, denoted by $y_0$. Preimages $\psi_0^{-1}(x_0)$ and $\psi_0^{-1}(y_0)$ are base points for domains associated to $N_{p_1 e^{q_1}}$, as $\psi_0$ is a bijection. The map $\psi$ is a homeomorphism, therefore the induced maps on the fundamental groups of the involved domains are isomorphisms. We can invoke Lemma~\ref{Lem:Lifting}, the unique lift $\psi_1$ of $\psi_0\circ N_{p_1 e^{q_1}}$ is a map from $\Chat\setminus T^1$ onto $\Chat\setminus V^2$ such that $\psi_1(\psi_0^{-1}(y_0))=y_0$ and $\psi_0\circ N_{p_1 e^{q_1}}=N_{p_2 e^{q_2}}\circ\psi_1$ on $\Chat\setminus T^1$.
 \[\begin{diagram}[heads=LaTeX,l>=3em]
 \Chat\setminus T^1  &\rTo^{\psi_1}   & \Chat\setminus T^2 \\
 & \rdTo_{\psi_0\circ N_{p_1 e^{q_1}}} & \dTo_{N_{p_2 e^{q_2}}}\\
 &   & \Chat\setminus V^2
 \end{diagram}\]
 We extend $\psi_1$ to the finite set $T^1$ as a continuous map. Observe that $\psi_1=\psi_0=\Psi$ on $\Omega^0$. 
 The unique lift $\tilde \psi_1$ of $\psi_0^{-1}\circ N_{p_2 e^{q_2}}$ is a map from $\Chat\setminus T^2$ onto $\Chat\setminus V^1$ such that $\tilde\psi_1(y_0)=\psi_0^{-1}(y_0)$ and
 $\psi_0^{-1}\circ N_{p_2 e^{q_2}}= N_{p_1 e^{q_1}}\circ\tilde\psi_0$ on $\Chat\setminus T^2$.
 \[\begin{diagram}[heads=LaTeX,l>=3em]
 \Chat\setminus T^1  & \lTo^{\psi_1}   				& \Chat\setminus T^2 \\
 \dTo_{N_{p_1 e^{q_1}}}			& \ldTo_{\psi_0^{-1}\circ N_{p_2 e^{q_2}}}  & \\
 \Chat\setminus V^1 & 				& 
 \end{diagram}\]
 Similarly, we extend $\tilde\psi_1$ to the finite set $T^2$ as a continuous map. It is easy to observe that $\psi_1$ and $\tilde\psi_1$ are inverses to each other on $\Chat$. Moreover, $\psi_1$ is a quasiconformal homeomorphism with the same complex dilatation as $\psi_0$. By continuing this lifting process, as lifts are carried out with holomorphic maps, we obtain a sequence of quasiconformal maps $\{\psi_m\}_{m\geq 0}$ with the same bound on complex dilatations such that for every $m\ge0$ we have $\psi_{m+1}=\psi_0=\Psi$ on $\Omega^0$ and $\psi_m\circ N_{p_1 e^{q_1}}=N_{p_2 e^{q_2}}\circ\psi_{m+1}$ on $\Chat$. Note that $\psi_m=\psi_{0}$ is conformal on $N_{p_1 e^{q_1}}^{\circ (-m)}(\Omega^1)$. The sequence $\{\psi_m\}_{m\geq 0}$ is a normal family, so it has a convergent sub-sequence, denote it by $\{\psi_{m_k}\}_{k\geq 0}$, and let $\psi_\infty$ be its limit. From the fact that the space of quasiconformal homeomorphisms with uniformly bounded dilatations is compact it follows that $\psi_\infty$ is quasiconformal. Note that, as constructed by lifts, the map $\psi_\infty$ is conformal on $\cup_{m=0}^{\infty} N_{p_1 e^{q_1}}^{\circ(-m)}(\Omega^1)$, the complement of which is a measure zero Julia set of $N_{p_1 e^{q_1}}$. Thus, $\psi_\infty$ is conformal on $\Chat$, a M\"obius transformation. We have $\psi_\infty\circ N_{p_1 e^{q_1}}=N_{p_2 e^{q_2}}\circ \psi_\infty$ on $\Omega^0$. 
 
 Consider a rational map $R:=\psi_\infty^{-1}\circ N_{p_2 e^{q_2}}\circ \psi_\infty$. Followed by the construction of the sequence of lifts, we have $R=N_{p_1 e^{q_1}}$ on $\Omega^0$. By the identity principle of holomorphic functions, we obtain $R=N_{p_1 e^{q_1}}$ on $\Chat$, i.e. $\psi_\infty\circ N_{p_1 e^{q_1}}=N_{p_2 e^{q_2}}\circ \psi_\infty$ on $\Chat$.
 The proof of surjectivity is finished here. 
 
\end{proof}
 
 \section{Proof of main theorem~\ref{Thm:Main}}
 Let us recall the definitions of the spaces with which we are working. For a pair of natural numbers $d\ge 3$ and $1\le n\le d$ , we have denoted by $\cN_{\text{pcm}}(d-n,n)$ the space of normalized postcritically minimal Newton maps $N_{p e^q}$ of degree $d\geq 3$ with $n$ petals at $\infty$. It was denoted by $\cN^{+,n}_\text{pcf}(d)$ the space of normalized postcritically \emph{finite} Newton maps with $n$-markings $\Delta^+_n$. Ha\"issinsky equivalence classes were denoted by $\sim_H$.
 
\begin{proof}[Proof of main theorem~\ref{Thm:Main}.]
 For a pair of natural numbers $d\ge3$ and $1\le n\le d$, define a map $\cF_{n}: \cN^{+,n}_\text{pcf}(d)/\sim_H \to \cN_\text{pcm}(d-n,n)/\text{\em Affine}$ induced by parabolic surgery, i.e. for every postcritically finite Newton map $N_p$ with $n$-marking $\Delta^+_n$ apply Theorem~\ref{Thm:Parabolic_Surgery_Newton}, which results to a postcritically minimal Newton map $N_{p_1 e^{q_1}}$ in $\cN_\text{pcm}(d-n,n)$, normalize $p_1$ and $q_1$ if necessary. The mapping $\cF_{n}$ is well defined, indeed by Theorem~\ref{Thm:Injectivity_Haissinsky} it follows that Ha\"issinsky equivalent classes of marked postcritically finite Newton maps produce affine conjugate results, moreover, a homeomorphism of the theorem preserves the dynamics and embedding of Julia sets. The mapping $\cF_{n}$ is also injective. Its surjectivity follows from Theorem~\ref{Thm:Surjectivity_Haissinsky}. We would like to remark that the parabolic surgery is a natural bijection in a sense that the dynamics and embedding of Julia set are preserved. It is also unique in the sense that different choices of perturbations of a postcritically minimal Newton maps result to the unique postcritically finite Newton map of a polynomial. Thus the correspondence is canonical. 
\end{proof}
 
 In \cite{LMS}, postcritically finite Newton maps of polynomials, of degree at least 3, have been classified in terms of connected finite graphs with certain properties. To get this finite data, one must consider an iterated preimage of the channel diagram along with an extended Hubbard tree of the periodic superattracting cycles of period greater than one. Newton rays connect the latter with the preimages of the former. If we include markings of channel diagram to the above data, then by Main Theorem~\ref{Thm:Main} a classification of postcritically minimal Newton maps of entire maps that take the form $p e^q$, for polynomials $p$ and $q$, becomes an easy corollary of the classification of postcritically finite Newton maps of polynomials.
 
 \subsection*{Acknowledgements}
 The author thanks Dierk Schleicher and the dynamics group at Jacobs University Bremen,
 especially Bayani Hazemach, Russell Lodge, and Sabyasachi Mukherjee, for their comments that helped to improve
 the manuscript version of the paper. The author would like also to show his gratitute to Guizhen Cui for providing his preprint and an anonymous referee for providing insightful comments. Research was partially supported by the ERC advanced grant ``HOLOGRAM" and the Deutsche Forschungsgemeinschaft SCHL 670/4.

\end{document}